\documentclass[a4paper]{gtart}

\usepackage[inner=20mm, outer=20mm, textheight=245mm]{geometry}

\usepackage{psfrag, graphicx, subfigure, epsfig,color}

\usepackage{amsmath, amssymb, latexsym, euscript}

\usepackage[colorlinks,citecolor=blue,linkcolor=blue, backref=page]{hyperref}
\usepackage[nameinlink]{cleveref}

\usepackage[margin=1cm]{caption}

\usepackage{tikz}
\usetikzlibrary{matrix}

\usepackage{mathptmx, wrapfig}

\usepackage{algorithmicx}
\usepackage{algpseudocode}

\usepackage{longtable}

\usepackage[colorinlistoftodos]{todonotes}
\DeclareMathOperator{\cross}{c}
\DeclareMathOperator{\SL}{SL}

\newcommand{\meridian}{\mathfrak{m}}
\newcommand{\longitude}{\mathfrak{l}}

\newcommand{\tri}{\mathcal T}

\theoremstyle{plain}
\newtheorem{theorem}{Theorem}
\newtheorem*{theorem*}{Theorem}
\newtheorem{lemma}[theorem]{Lemma}
\newtheorem{proposition}[theorem]{Proposition}
\newtheorem{corollary}[theorem]{Corollary}

\theoremstyle{definition}

\newtheorem*{definition*}{Definition}

\newtheorem{algorithm}[theorem]{Algorithm}

\theoremstyle{remark}
\newtheorem{remark}[theorem]{Remark}

\numberwithin{equation}{section}

\begin{document}

\title{Slope norm and an algorithm to compute the crosscap number}
\author{William Jaco, J. Hyam Rubinstein, Jonathan Spreer and Stephan Tillmann}

\begin{abstract} 
We give three algorithms to determine the crosscap number of a knot in the 3--sphere using $0$--efficient triangulations and normal surface theory. Our algorithms are shown to be correct for a larger class of complements of knots in closed 3--manifolds. The crosscap number is closely related to the minimum over all spanning slopes of a more general invariant, the slope norm. For any irreducible 3--manifold $M$ with incompressible boundary a torus, we give an algorithm that, for every slope on the boundary that represents the trivial class in $H_1(M; \mathbb{Z}_2)$, determines the maximal Euler characteristic of any properly embedded surface having a boundary curve of this slope. We complement our theoretical work with an implementation of our algorithms, and compute the crosscap number of knots for which previous methods would have been inconclusive. In particular, we determine 196 previously unknown crosscap numbers in the census of all knots with up to 12 crossings.
\end{abstract}


\primaryclass{57K31, 57K10, 57K32}
\keywords{3--manifold, knot, efficient triangulation, crosscap number, knot genus, boundary slope, slope norm}
\makeshorttitle


\section{Introduction}
\label{sec:intro}

The main contribution of this paper are algorithms to compute the crosscap number of a knot in the 3--sphere.
Efforts to compute the crosscap number of a knot have been at the centre of various other research projects using a variety of techniques. Among these is a formula for the crosscap number of torus knots by Teragaito \cite{Teragaito04CCTorusKnots}; an algorithm for alternating knots developed by Adams and Kindred \cite{Adams13CrosscapNoForAlternatingKnots}; or upper and lower bounds for the general case via the Jones polynomial by Kalfagianni and Lee~\cite{Kalfagianni16BoundsForCrosscapNumber}. Recent work by Ito and Takimura~\cite{Ito18CCKnotProjections,Ito20CCAndVolumeBounds,Ito20CCBoundsAlternatingKnots} establishes various further bounds. The \texttt{KnotInfo} data base \cite{knotinfo}, and in particular their page on crosscap numbers, gives a detailed overview and results for specific knots.

Our work completes an approach put forward by Burton and Ozlen~\cite{burton12-crosscap}.
Our starting point is Jaco and Sedgwick's generalisation~\cite{Jaco-decision-2003} of a celebrated result by Hatcher~\cite{Hatcher-boundary-1982}: they showed that in any orientable, irreducible 3--manifold with incompressible boundary a torus there are only finitely many boundary slopes of \emph{geometrically} incompressible and $\partial$-incompressible surfaces. 
This paper rests on a technical observation in \cite{Jaco-decision-2003} that concerns fundamental surfaces (stated here as \Cref{cor:JS-key}). We refer the reader to either of \cite{burton12-crosscap, jaco03-0-efficiency, Jaco-decision-2003, Matveev-algorithmic-2007} for the definitions and basic properties of normal surface theory in (singular) triangulations; 
\cite{burton12-crosscap, jaco03-0-efficiency} for basic facts concerning $0$--efficient triangulations used in this paper; and \cite{Burton-computing-2018, Tillus-normal-2008, tollefson98-quadspace}
for basic properties of working with quadrilateral coordinates only.

In \Cref{sec:algo,sec:crosscap,sec:knot genus}, we first develop our algorithms using standard coordinates for normal surfaces under varying hypotheses on the triangulations. We then extend the theory to work in quadrilateral coordinates in \Cref{sec:quad_space}, and report on our computational results within this framework in \Cref{sec:computation}. Throughout this paper, a \textbf{fundamental surface} is a normal surface whose normal coordinate is fundamental in standard (triangle-quadrilateral) normal surface space; and a \textbf{$\mathbf{Q}$--fundamental surface} is a connected normal surface whose normal $Q$--coordinate is fundamental in quadrilateral space.

\textbf{Slope norm.} The dual tree to the Farey tesselation of the hyperbolic plane is used to organise the set of all boundary slopes of properly embedded surfaces with a single boundary curve. This allows us to give an algorithm that, for an irreducible $3$-manifold with boundary a torus and a slope on the boundary that represents the trivial class in $H_1(M; \mathbb{Z}_2)$, determines the maximal Euler characteristic of any properly embedded connected surface having connected boundary of this slope (\Cref{thm:slope norm-algo}). We call the negative of this number the \emph{norm of the slope}, and the minimum over all these norms the \emph{slope norm} of $M$. 
This norm is used in forthcoming work to apply the complexity bounds given in \cite{Jaco-norm-2020, Jaco-ideal-2020, Jaco-minimal-2009} to infinite families of Dehn fillings. The existence of such an algorithm goes back to Schubert~\cite{Schubert1961-bestimmung}, see also Matveev~\cite[Theorems 4.1.10 \& 4.1.11]{Matveev-algorithmic-2007}. Our new contribution is that we do not need to adapt a triangulation to the slope.
Similar results to our slope norm algorithm were obtained independently by Howie~\cite{Howie-geography}, and used for different applications. In this paper, our work on slope norm feeds into the proof of the main theorem. We later show that there is an algorithm to determine the slope norm of $M$ and the set of all minimising slopes in quadrilateral space (\Cref{cor:slope norm-Q}).

\textbf{Crosscap number.} We next give two algorithms to determine the crosscap number of a knot in the 3--sphere (\Cref{thm:crosscap-suitable,thm:crosscap}). Both use standard coordinates for normal surfaces, but they make different assumptions on the underlying triangulation. 
Our algorithms are shown to be correct for a larger class of complements of knots in closed 3--manifolds $N$ that represent the trivial class in $H_1(N; \mathbb{Z}_2)$, including all that have a complete hyperbolic structure of finite volume.

Burton and Ozlen~\cite{burton12-crosscap} introduce triangulations that contain no normal 2-spheres and have an edge in the boundary that represents the meridian. These triangulations are called \emph{efficient suitable}, and they guarantee the existence of fundamental spanning surfaces of maximal Euler characteristic. Burton and Ozlen describe a procedure (Algorithm 3), where on input a knot in the 3--sphere the output is either one integer (the crosscap number) or a pair of consecutive integers (one of which is the crosscap number). \Cref{thm:crosscap-suitable} shows that in the latter case, the crosscap number is the larger integer. The apparent difficulty of determining the crosscap number algorithmically lies in the case, where every maximal Euler characteristic fundamental spanning surface is orientable. This is solved in \Cref{lem:main_lemma}. The following result thus improves and generalises \cite[Algorithm 3]{burton12-crosscap}.
\begin{theorem}
  \label{thm:crosscap-suitable}
Let $M$ be the exterior of a non-trivial knot $K$ in a closed 3--manifold $N$ with $[K] = 0 \in H_1(N; \mathbb{Z}_2).$ Suppose that $M$ is irreducible and contains no embedded non-separating torus and no embedded Klein bottle. Let $\tri$ be an efficient suitable triangulation of $M$. Then $\cross(K)=\min (A, B),$ where
\begin{itemize}
\item A = $\min \{ \; 1 - \chi(S) \; \mid S \text{ is a non-orientable fundamental spanning surface for } K \; \}$
\item B = $\min \{ \; 2 - \chi(S) \; \mid S \text{ is an orientable fundamental spanning surface for } K \; \}$
\end{itemize}
and we let $\min \emptyset = \infty.$
\end{theorem}

\Cref{thm:crosscap-suitable} has the following consequence in the case where a knot has no orientable spanning surface:

\begin{corollary}
  \label{cor:crosscap-suitable}
Let $M$ be the exterior of a non-trivial knot $K$ in a closed 3--manifold $N$ with $[K] = 0 \in H_1(N; \mathbb{Z}_2)$ and $[K] \neq 0 \in H_1(N; \mathbb{Z}).$
Suppose that $M$ is irreducible and contains no embedded non-separating torus and no embedded Klein bottle. Let $\tri$ be an efficient suitable triangulation of $M$. Then 
\[
\cross(K) = \min \{ \; 1 - \chi(S) \; \mid S \text{ is a fundamental spanning surface for } K \; \}
\]
\end{corollary}

For arbitrary 0--efficient triangulations (which do not contain properly embedded normal spheres or discs), we give a more general algorithm (\Cref{thm:crosscap}) that uses the slope norm algorithm. 
The basic idea is that Burton and Ozlen's suitable triangulations ensure that minimal spanning surfaces can be found amongst the normal surfaces even though in general they may be $\partial$--compressible. In an arbitrary triangulation, there may be no non-orientable normal spanning surfaces of maximal Euler characteristic, but our slope norm algorithm keeps track of optimal boundary compression sequences by determining the shortest path from a given slope to the subtree $\mathcal{F}_e$ corresponding to the slopes of spanning surfaces.

\begin{theorem}
  \label{thm:crosscap}
Let $M$ be the exterior of a knot $K$ in a closed 3--manifold $N$ with $[K] = 0 \in H_1(N; \mathbb{Z}_2).$ Suppose that $M$ is irreducible and contains no embedded non-separating torus and no embedded Klein bottle. Let $\tri$ be a $0$-efficient triangulation of $M$ and suppose that the coordinates for a meridian for $K$ on the induced triangulation $\tri_\partial$ of $\partial M$ are given. Then $\cross(K)=\min (A, B, Z),$ where
\begin{itemize}
\item A = $\min \{ \; 1 - \chi(S) \; \mid S \text{ is a non-orientable fundamental spanning surface for } K \; \}$
\item B = $\min \{ \; 2 - \chi(S) \; \mid S \text{ is an orientable fundamental spanning surface for } K \; \}$
\item Z = $\min \{ \; 1 - \chi(S) + d(\partial S, \mathcal{F}_e) \; \mid S \text{ is a fundamental non-spanning surface}$
\item[] \hspace{10cm}$\text{with connected essential boundary} \; \}$
\end{itemize}
and we let $\min \emptyset = \infty.$
\end{theorem}

Since $d(\partial S, \mathcal{F}_e)=0$ for a spanning surface $S,$ the above theorem could have been stated using two terms rather than three, but we wanted to keep the notation in line with \Cref{thm:crosscap-suitable}.
Again, in the case where there is no orientable spanning surface, this specialises to:

\begin{corollary}
  \label{cor:crosscap}
Let $M$ be the exterior of a non-trivial knot $K$ in a closed 3--manifold $N$ with $[K] = 0 \in H_1(N; \mathbb{Z}_2)$ and $[K] \neq 0 \in H_1(N; \mathbb{Z}).$
Suppose that $M$ is irreducible and contains no embedded non-separating torus and no embedded Klein bottle. Let $\tri$ be a $0$-efficient triangulation of $M$ and suppose that the coordinates for a meridian for $K$ on the induced triangulation $\tri_\partial$ of $\partial M$ are given. 
Then 
\[
\cross(K) = \min \{ \; 1 - \chi(S) + d(\partial S, \mathcal{F}_e) \; \mid S \text{ is a fundamental surface with connected essential boundary }  \}
\]
\end{corollary}

\textbf{Knot genus.} As an interlude, we show in \Cref{sec:knot genus} that in the framework of this paper, one can give a short proof of Schubert's classical result that the genus of a knot is realised by one of the fundamental surfaces. Schubert~\cite{Schubert1961-bestimmung} originally proved this in the context of normal surfaces with respect to handle decompositions.

\begin{theorem}
  \label{thm:knot-genus}
Let $M$ be the exterior of a non-trivial knot $K$ in a closed 3--manifold $N$ with $[K] = 0 \in H_1(N; \mathbb{Z}).$ Suppose that $M$ is irreducible and let $\tri$ be a 0--efficient triangulation of $M$. Then an orientable spanning surface of maximal Euler characteristic is amongst the fundamental surfaces.
\end{theorem}

\textbf{Quadrilateral space.} All results up to this point were stated in the context of normal surface theory with standard coordinates. For computations, it is of advantage to be able to work with quadrilateral coordinates only as this makes otherwise currently impossible calculations feasible. In \Cref{sec:quad_space}, we give some extensions of the previous results in this context. See for instance \cite{Burton-computing-2018} for similar results for closed normal surfaces. For definitions and basic properties of working with quadrilateral coordinates only, we refer to \cite{Burton-computing-2018, Tillus-normal-2008, tollefson98-quadspace}. The following is the main result of this paper; whilst the previous results were given for either efficient suitable triangulations or for 0--efficient triangulations, we now need to combine these properties.

\begin{theorem}
\label{thm:Q-algo}
Let $M$ be the exterior of a non-trivial knot $K$ in a closed 3--manifold $N$ with $[K] = 0 \in H_1(N; \mathbb{Z}_2).$ Suppose that $M$ is irreducible and contains no embedded non-separating torus and no embedded Klein bottle. Let $\tri$ be a $0$-efficient suitable triangulation of $M.$ 
Then $\cross(K)=\min (A', B'),$ where
\begin{itemize}
\item $A' = \min \{ \; 1 - \chi(S) \; \mid S \text{ is a non-orientable $Q$--fundamental spanning surface for } K \; \}$
\item $B' = \min \{ \; 2 - \chi(S) \; \mid S \text{ is an orientable $Q$--fundamental spanning surface for } K \; \}$
\end{itemize}
\end{theorem}

\textbf{Computations.} We use our methods to determine the crosscap numbers of 196 knots with up to 12 crossings for which the crosscap number was previously not known. As a result, crosscap numbers of all knots up to ten crossings are now known. Our algorithms give a \emph{theoretical} method to determine the crosscap numbers of knots with a unified method, where previously different techniques were needed. From a \emph{practical} viewpoint, using our algorithm in standard coordinates allows us to handle triangulations of up to about 26 tetrahedra. Making use of our results in quadrilateral space allows us to push this limit to about 30 tetrahedra. See \Cref{sec:computation} for our computational results.

\textbf{Acknowledgements.} 
The authors thank Josh Howie for comments on an earlier draft.
Jaco is partially supported by the Grayce B. Kerr Foundation.
Research of Rubinstein, Spreer and Tillmann is supported in part under the Australian Research Council's Discovery funding scheme (project number DP190102259). 


\section{The norm of an even slope}
\label{sec:algo}

Throughout this section let $M$ be an orientable, compact, irreducible 3--manifold with $\partial M$ a single, incompressible torus. Suppose $\tri$ is a (singular or semi-simplicial) triangulation of $M$ with the property that the induced triangulation $\tri_\partial$ of $\partial M$ has exactly one vertex. For instance, a 0--efficient triangulation has this property~\cite{jaco03-0-efficiency}. In the following, we choose the single vertex of $\tri$ as the base point for the fundamental group of $M$ and omit it in the notation.

We first show that every connected essential curve $c\subset\partial M$ with $[c] = 0\in H_1 (M,\mathbb{Z}_2)$ bounds a properly embedded connected surface $S \subset M$ with $\partial S = c$ (\Cref{cor:evenslopes}). We then use a result by Jaco and Sedgwick, \Cref{cor:JS-key} or \cite[Corollary 3.8]{Jaco-decision-2003}, to conclude that if $\tri$ is 0--efficient and $S$ is 
of maximal Euler characteristic amongst all connected surfaces with this slope and $\partial$-incompressible, then a surface of equal boundary slope and Euler characteristic must be represented by a fundamental surface in $\tri$. This then leads to \Cref{algo:slopenorm} to compute the smallest norm of a given boundary slope on $\partial M$ with respect to a given framing.

\begin{lemma}
\label{lem:Z2-hom-dual-surface}
Let $\rho \co \pi_1(M)\to \mathbb{Z}_2$ be a homomorphism with the property that $\rho(\meridian) = 1$ for some primitive peripheral element $\meridian \in \text{im}(\pi_1(\partial M)\to \pi_1(M)).$ Then for every primitive peripheral element $\gamma \in \ker(\rho),$ there is a properly embedded surface $S$ in $M$ with $\partial S$ a single boundary curve that satisfies $[\partial S] = \gamma^{\pm 1}$ as free homotopy classes of unoriented loops. In particular, $[\gamma] = 0 \in H_1(M, \mathbb{Z}_2).$
\end{lemma}

The proof of the lemma introduces the way we will use the Farey graph in our later algorithms. 

\begin{proof}
First note that there is $\longitude \in \text{im}(\pi_1(\partial M)\to \pi_1(M))$ with $\rho(\longitude) = 0$ and $\langle \meridian, \longitude\rangle=\text{im}(\pi_1(\partial M)\to \pi_1(M)).$ The primitive peripheral elements in $\ker \rho$ are precisely the elements $\meridian^{2k}\longitude^q,$ where $k \in \mathbb{Z}$ and $\gcd(2k,q)=1.$ Since the boundary of $M$ is incompressible and has abelian fundamental group, we have 
\begin{equation}\label{eq:peripheral}
\text{im}(\pi_1(\partial M)\to \pi_1(M))\cong \pi_1(\partial M) \cong H_1(\partial M, \mathbb{Z}) 
\end{equation}
We therefore identify all groups and freely switch between additive and multiplicative notation for peripheral elements. Since we are only interested in unoriented isotopy classes of primitive elements, we always choose $\meridian^p \longitude^q$ with $q\ge 0.$

We first show that there is a surface for some peripheral element in the kernel. Our 0--efficient triangulation $\tri$ of $M$ has exactly one vertex. Hence every edge is a loop and represents an element of $\pi_1(M)$. This maps to either $0$ or $1$ under $\rho.$ Place a normal corner on an edge if and only if the corresponding element maps to $1.$ As observed in \cite{Jaco-minimal-2009}, this results in a normal surface in $M$ having at most a single triangle or a single quadrilateral in each tetrahedron as shown in \Cref{fig:canonical_surface}. Since $\rho(\meridian) =1,$ this normal surface has a single boundary curve $[\partial S] = \meridian^{2k'} \longitude^{q'}$ for some $q'\ge 0$ and $\gcd(2k',q')=1.$ This single boundary curve meets each boundary triangle in a single normal arc.

    \begin{figure}[h]
    \centerline{\includegraphics[width=\textwidth]{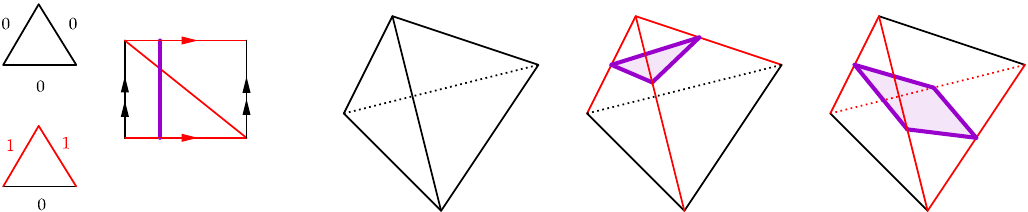}}
    \caption{Labelling of edges and canonical normal curves and surfaces.}
    \label{fig:canonical_surface}
  \end{figure}    

We now show how all other boundary curves $\meridian^{2k}\longitude^q,$ where $q>0,$ $k \in \mathbb{Z}$ and $\gcd(2k,q)=1$ can be obtained by adding saddles to $S.$ To this end, we use the \emph{layering procedure} (see \cite{Jaco-minimal-2009} and \Cref{fig:layerings2}) in conjunction with the Farey tesselation as an organising principle for the set of isotopy classes of triangulations of the torus with a marked point. This is different from the $L$--graph used in \cite{Jaco-minimal-2009}.

We treat the single vertex in the induced triangulation $\tri_\partial$ of $\partial M$ as a marked point, and give the torus $\partial M$ a Euclidean structure with the property that the marked point lifts to the integer lattice via a universal covering map $\mathbb{R}^2 \to \partial M.$ Moreover, up to the action of the Deck group, we may assume that $\meridian$ and $\longitude$ lift to horizontal and vertical lines respectively. The map 
\[ \meridian^p \longitude^q \ \mapsto \frac{p}{q} \]
gives a bijection between the set of isotopy classes of primitive curves and $\mathbb{Q} \cup \{ \infty \}.$ Now the isotopy classes of triangulations with single vertex at the marked point correspond to the orbit of the triple $\big( \frac{1}{0}, \frac{0}{1}, \frac{-1}{1} \big)$ under the action of $\SL(2,\mathbb{Z})$ by M\"obius transformations. 
Identifying $\mathbb{R} \cup \{ \infty \}$ with the boundary of the Poincar\'e disc model of the hyperbolic plane and each triple with an ideal triangle gives the well-known Farey tesselation.

Each triple $\big( \frac{p_0}{q_0},$ $\frac{p_1}{q_1},$ $\frac{p_2}{q_2} \big)$ contains precisely one fraction with even numerator, say $\frac{p_0}{q_0}.$ Note that the corresponding primitive element is in the kernel of $\rho,$ whilst the other two are mapped to $1.$ The normal curve resulting from our above procedure of assigning $0$ or $1$ to each edge, when applied to the corresponding marked triangulation of the boundary, results in a normal curve of slope $\meridian^{p_0}\longitude^{q_0}.$ We call $\frac{p_0}{q_0}$ the \textbf{even slope} of the triple.

 \begin{figure}[h]
    \centerline{\includegraphics[width=0.6\textwidth]{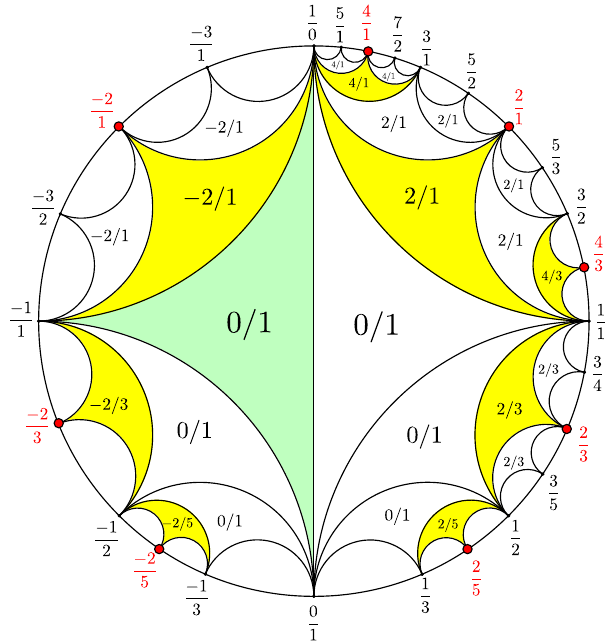}}
    \caption{The Farey tesselation. Each ideal triangle corresponds to an isotopy class of 1-vertex triangulations of the torus. The ideal vertices are labelled with the slopes of the edges, and each ideal triangle is labelled with its unique even slope. Marked are the base triangle in green, and the canonical triangles for the even slopes in yellow. Adjacent triangles differ by an edge flip.}
    \label{fig:Farey}
  \end{figure}  

The dual 1--skeleton to the Farey tesselation is an infinite trivalent tree. Any two triangulations that correspond to triangles sharing an edge in the tesselation are related by an \emph{edge flip}. If one of the triangulations corresponds to $\tri_\partial,$ then, as an operation on the triangulation $\tri$ of $M$, one can \emph{layer} a tetrahedron $\sigma$ on $\tri$ along the edge that is being flipped; see \Cref{fig:layerings2}. This results in a new triangulation $\tri' = \tri \cup \sigma$ of $M$ with the property that the isotopy class of the triangulation of the boundary has changed. Since the trivalent tree is connected, one observes that all isotopy classes of triangulations of $\partial M$ can be realised as the induced triangulations of the boundary of $M$. In particular, every even slope is an edge in some triangulation of the boundary of $M.$ It remains to relate this information to triangulations of $M$ and boundary curves of properly embedded surfaces.

\begin{figure}[h]
    \centerline{\includegraphics[width=0.4\textwidth]{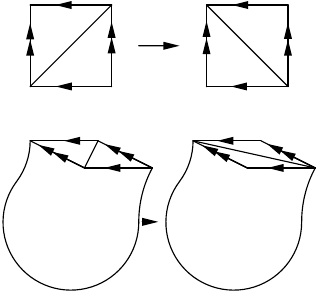}}
    \caption{Layering on a boundary edge adds a tetrahedron to the triangulation of $M.$}
    \label{fig:layerings2}
  \end{figure}   

Each triangulation of $\partial M$ allows three distinct flips, corresponding to the three edges of the ideal triangle in the Farey tesselation. The effect of layering a tetrahedron on the triangulation has the effect of changing the normal surface $S$ to a normal surface $S'$ by adding a quadrilateral. If the corresponding ideal edge in the Farey tesselation has endpoint at the even slope, then the layering adds a pinched annulus to $S$ and maintains its boundary slope. If the ideal edge does not have an endpoint at the even slope, then a punctured and pinched M\"obius strip (which we call a \textbf{saddle}) is added to $S$, hence its Euler characteristic lowered by 1, and its slope changes to the even slope of the adjacent triangle in the Farey tesselation. Topologically, the relationship between the surfaces is that for two of the three layerings, $S$ is obtained from $S'$ by deleting a pinched annulus, whilst for the last, $S$ is obtained from $S'$ by performing a boundary compression. The three possibilities are shown in \Cref{fig:layerings1}.

 \begin{figure}[h]
    \centerline{\includegraphics[width=0.7\textwidth]{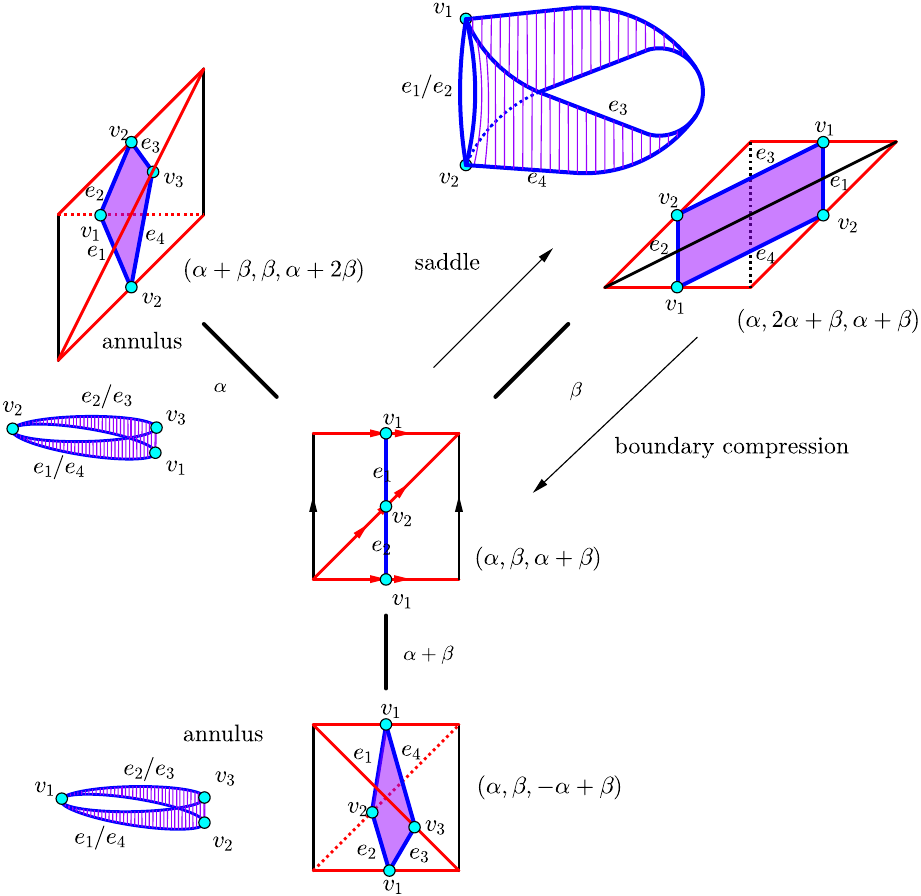}}
    \caption{The three possible layerings on the boundary torus. Two add a pinched annulus to the surface and do not change the boundary slope. The last adds a saddle and changes the boundary slope according to the labelling of the associated ideal triangles in the Farey graph.}
    \label{fig:layerings1}
  \end{figure}   

This completes the proof of the lemma, since starting with $S$, for every even slope, this constructs a properly embedded surface with boundary that slope. This construction and our observations about the Farey tesselation are key to the algorithm given in \Cref{thm:slope norm-algo}.
\end{proof}

The inclusion map induces a homomorphism $H_1(\partial M, \mathbb{Z}_2) \to H_1(M, \mathbb{Z}_2)$ which we precompose with the natural map $H_1(\partial M, \mathbb{Z})\to H_1(\partial M, \mathbb{Z}_2)$ to obtain $\varphi\co H_1(\partial M, \mathbb{Z}) \to H_1(M, \mathbb{Z}_2).$ The next lemma shows that the homomorphism in the hypothesis of \Cref{lem:Z2-hom-dual-surface} exists and is unique. 

\begin{lemma}
\label{lem:2-basis}
We have $H = \text{im} ( \varphi\co H_1(\partial M, \mathbb{Z}) \to H_1(M, \mathbb{Z}_2)) \cong \mathbb{Z}_2.$ 
\end{lemma}

It follows from the lemma that we may choose a basis $( \meridian_2, \longitude_2)$ of $H_1(\partial M, \mathbb{Z})$ with the property that 
$\varphi(\meridian_2) \neq 0 = \varphi( \longitude_2).$ We say that the basis $( \meridian_2, \longitude_2)$ is a \textbf{2--torsion framing} of $\partial M$ and call $\meridian_2$ a \textbf{2--meridian} and $\longitude_2$ a \textbf{2--longitude}. 

\begin{proof}[First proof (geometric topology).]
First assume that $H = \text{im} ( \varphi\co H_1(\partial M, \mathbb{Z}) \to H_1(M, \mathbb{Z}_2)) =  \{0\}.$
Let $( \meridian, \longitude)$ be any basis of $H_1(\partial M, \mathbb{Z}).$ Choose a sufficiently fine simplicial triangulation of $M$ so that we may choose a simple closed curve $L$ in the 1--skeleton in $\partial M$ that is isotopic to $\longitude.$ Since $\varphi(\longitude)=0,$ there is a 2--chain $C$ in the 2--skeleton with $\partial C = L.$ Since we are working with $\mathbb{Z}_2$ coefficients, $C$ is an assignment of $0$ or $1$ to each 2--simplex in the triangulation. We add a product collar to $\partial M$ and add the annulus $L \times [0,1]$ to $C.$ We let $M$ and $C$ denote the resulting manifold and chain again.

Each 1--simplex in the interior of $M$ meets $C$ in an even number of 2--simplices. Hence, away from the vertices, we can resolve the 2--simplices in pairs to obtain a properly embedded but possibly singular surface $S'$ in $M$ with the property that its singularities are contained in the set of interior vertices of the triangulation of $M$. Now a small regular neighbourhood $N$ of the union of all vertices meets $S'$ in a union of circles. Hence, we replace $S\cap N$ with a union of discs, giving a properly embedded surface $S$ in $M$ with $\partial S = L.$ 

Since $\meridian$ meets $\longitude$ in a single point, the intersection pairing implies that $\varphi(\meridian_\infty) \neq 0 \in H_1(M, \mathbb{Z}_2).$ But this contradicts our hypothesis that $H= \{0\}.$

Hence $H \neq \{0\}.$ It now follows from \Cref{lem:Z2-hom-dual-surface} that the rank of $H$ cannot be two. Hence, it must be one.
\end{proof}

\begin{proof}[Second proof (algebraic topology).]
First assume that $H = \text{im} ( \varphi\co H_1(\partial M, \mathbb{Z}) \to H_1(M, \mathbb{Z}_2)) =  \{0\}.$ 
Glue a solid torus to $M$, resulting in a closed 3--manifold $N.$ Since $N$ is closed, the Euler characteristic of $N$ is zero and we have $b_1(N, \mathbb{Z}_2) = b_2(N, \mathbb{Z}_2).$ Consider the following part of the Mayer-Vietoris long exact sequence in homology with $\mathbb{Z}_2$ coefficients:
\[ \ldots \to H_1(\partial M, \mathbb{Z}_2) \to H_1(M, \mathbb{Z}_2) \oplus H_1(S^1\times D^2, \mathbb{Z}_2) \to H_1(N, \mathbb{Z}_2) \to \ldots.\]
Since $H =  \{0\},$ we have $H_1(N, \mathbb{Z}_2) \cong H_1(M, \mathbb{Z}_2).$ 
Now $H =  \{0\}$ also implies that there is a relative $\mathbb{Z}_2$--chain in $M$ that attaches to the meridian disc of the solid torus. The intersection pairing with the core curve of the solid torus implies that the rank of $H_2(N, \mathbb{Z}_2)$ is one larger than the rank of $H_2(M, \mathbb{Z}_2).$
 In particular, $b_1(M, \mathbb{Z}_2)= b_2(M, \mathbb{Z}_2)+1.$

Now consider the long exact sequence for the pair $(M, \partial M)$ with $\mathbb{Z}_2$--coefficients. We obtain
\[ 0 \to H_2(M, \mathbb{Z}_2) \to H_2(M, \partial M, \mathbb{Z}_2) \to H_1(\partial M, \mathbb{Z}_2) \to 0\]
Using Poincar\'e-Lefschetz duality and the universal coefficient theorem, we have
\[ H_1(M, \mathbb{Z}_2) \cong H^1(M, \mathbb{Z}_2) \cong H_2(M, \partial M, \mathbb{Z}_2) \cong H_2(M, \mathbb{Z}_2) \oplus H_1(\partial M, \mathbb{Z}_2)\]
This gives $b_1(M, \mathbb{Z}_2) = b_2(M, \mathbb{Z}_2)+2,$ contradicting the calculation in the first paragraph.

Hence $H \neq \{0\}.$ It now follows from \Cref{lem:Z2-hom-dual-surface} that the rank of $H$ is one.
\end{proof}

\begin{remark}
\label{def:Homological_framing}
The standard half-lives half-dies argument \cite[Lemma 3.5]{HatcherNotes} implies for homology with rational coefficients that one may choose a basis $\langle \meridian_\infty, \longitude_\infty\rangle = H_1(\partial M, \mathbb{Z})$ with the property that $\meridian_\infty$ maps to an element of infinite order whilst $\longitude_\infty$ maps to an element of finite order under the inclusion map to $H_1(M, \mathbb{Z}).$ In particular, $\longitude_\infty$ is uniquely determined up to sign, whilst $\meridian_\infty$ is only well-defined up to sign and a power of $\longitude_\infty.$ We call $\longitude_\infty$ \textbf{the homological longitude} and $\meridian_\infty$ \textbf{a homological meridian}. 
\end{remark}

It is not necessarily the case that one may choose $( \meridian_2, \longitude_2) = ( \meridian_\infty, \longitude_\infty).$ To make this statement less mysterious, we give a third proof of the lemma that does not appeal to a contradiction.

\begin{proof}[Third proof (geometric topology).]
Suppose the order of  $\longitude_\infty$ is $m$ in $H_1(M, \mathbb{Z}).$ The significance of the order is that $\meridian_\infty$ maps to an element of the form $a^mh \in H_1(M, \mathbb{Z}),$ where $a$ generates a free $\mathbb{Z}$ summand and $h$ is a torsion element. A geometric interpretation of this algebraic relationship arises from a construction due to Stallings~\cite{Stallings-fibering-1962} that produces a properly embedded, connected  oriented surface $S$ in $M$ with $[\partial S] = \longitude^{\pm m}_\infty$ dual to the action of $\pi_1(M)$ on $\mathbb{R}$ associated with a homomorphism $\pi_1(M) \to H_1(M, \mathbb{Z}) \to \mathbb{Z}$ with $a \mapsto 1.$
Moreover, $S$ has exactly $m$ boundary components, which implies that all have the same induced orientation, and the meridian has algebraic intersection number $\pm m$ with the surface.

Note that if one connects any two adjacent boundary components with a boundary parallel annulus, then one obtains a non-orientable surface. In particular, if $m$ is odd, then one may connect pairs of boundary components with annuli to obtain a properly embedded non-orientable surface $S$ in $M$ with a single boundary curve $[\partial S] = \longitude^{\pm 1}_\infty.$ In particular, $\meridian_\infty$ maps to a generator of the image and $\longitude_\infty$ is contained in the kernel of $\varphi\co H_1(\partial M, \mathbb{Z}) \to H_1(M, \mathbb{Z}_2).$ Hence we may choose $\meridian_2 = \meridian_\infty$ and $\longitude_2 = \longitude_\infty.$ 

If $m$ is even, the same construction of connecting boundary components in pairs results in a closed non-orientable surface in $M.$ Since $m$ is the order of $\longitude_\infty$ in $H_1(M, \mathbb{Z}),$ there are two cases, depending on whether $\varphi( \longitude_\infty)$ maps to zero or not.

First assume that $m=2$ and $\varphi( \longitude_\infty)$ is the generator of a $\mathbb{Z}_2$--summand. Then there is a homomorphism $\rho\co \pi_1(M) \to \mathbb{Z}_2$ with  $\rho(\longitude_\infty)= 1.$ Then
$\rho(\gamma)= 0,$ where either $\gamma = \meridian_\infty$ or $\gamma = \meridian_\infty\longitude_\infty.$
Now $[\gamma] = 0 \in H_1(M, \mathbb{Z}_2)$ according to \Cref{lem:Z2-hom-dual-surface}. It follows that we may choose $\meridian_2 = \longitude_\infty$ and $\longitude_2 = \gamma.$

The remaining case is that $\varphi( \longitude_\infty)=0.$ In this case, the construction from the first proof of the surface $S$ with $[\longitude_\infty] = [\partial S] = 0 \in H_1(M, \mathbb{Z}_2)$ can be applied, and we let $\longitude_2 = \longitude_\infty.$ Since $\meridian_\infty$ meets $\longitude_\infty$ in a single point, we have $\varphi(\meridian_\infty) \neq 0 \in H_1(M, \mathbb{Z}_2)$. So we let $\meridian_2 = \meridian_\infty.$
\end{proof}

We use the following terminology and notation for \textbf{unoriented isotopy classes} of non-trivial simple closed loops on the boundary torus. 
Let $\alpha \in \text{im}(\pi_1(\partial M)\to \pi_1(M)).$ Recalling the identification \Cref{eq:peripheral},
consider $[\alpha] \in H_1(\partial M, \mathbb{Z}).$
If $[\alpha] = \meridian^p_2 \longitude^q_2$ is a non-trivial primitive class in $H_1(\partial M, \mathbb{Z})$ with $q\ge 0,$ then we call $\alpha$ a \textbf{slope}. We may therefore identify a slope with an unoriented isotopy class of a non-trivial simple closed loop on the torus. 
Conversely, each such unoriented isotopy class arises from a unique slope.
A slope $\alpha$ is an \textbf{even slope} if $\alpha$ maps to zero in $H_1(M, \mathbb{Z}_2).$ We remark that this is consistent with the terminology concerning even slopes in the Farey construction in \Cref{lem:Z2-hom-dual-surface}, and that the notion of an even slope is independent of the chosen 2--torsion framing.

Given a surface $S$ with connected boundary, we give $\partial S$ the unique orientation that makes $[\partial S]\in \pi_1(M)$ a slope.
Now \Cref{lem:Z2-hom-dual-surface,lem:2-basis} imply:

\begin{corollary}
\label{cor:evenslopes}
Let $\alpha \in \text{im}(\pi_1(\partial M)\to \pi_1(M))$ be a slope.
There is a properly embedded surface $S$ in $M$ with $[\partial S] = \alpha$ if and only if $\alpha$ is an even slope.
\end{corollary}

The \textbf{norm} of an even slope $\alpha$ is defined as:
\[ || \; \alpha \; || = \min \{\ - \chi(S) \ \mid \ \text{$S$ is a properly embedded surface in $M$ with }  [\partial S] = \alpha \} \]
Since $\partial M$ is incompressible, we have $|| \; \alpha \; || \ge 0$ for all even slopes. We say that $S$ is \textbf{taut} for $\alpha$ if $S$ is connected, $[\partial S] = \alpha$ and $|| \; \alpha \; || = - \chi(S)$.

\begin{remark}
We emphasise here that the slope norm is defined on isotopy classes of simple closed connected curves on the boundary of $M$, and that it gives the maximal Euler characteristic of a surface with connected boundary of this slope. A related approach, which we do not take, would be to also consider surfaces with multiple boundary components.
\end{remark}

The \textbf{slope norm} of $M$ is defined as:
\[ || \; M \; || = \min \{\ || \; \alpha \; || \ \mid \ \text{$\alpha$ is an even slope }  \} \]
Each even slope $\alpha$ satisfying $|| \; \alpha \; || =  || \; M \; ||$ is called a \textbf{minimising slope} for $M.$

The key results that will be used from Jaco and Sedgwick's work are~\cite[Proposition 3.7 and Corollary 3.8]{Jaco-decision-2003}. We require a few definitions in order to state them. We refer readers unfamiliar with the concepts in normal surface theory that are mentioned here to either of \cite{jaco03-0-efficiency,Jaco-decision-2003}.

The boundary curves of a normal surface $S$ form a collection of normal curves on $\partial M.$ It is shown in \cite{Jaco-decision-2003} that two normal curves are normally isotopic with respect to $\tri_\partial$ if and only if they are isotopic on $\partial M$, and that a normal curve is trivial if and only if it is vertex linking. It follows that we can identify a slope with the normal isotopy class of a non-trivial normal curve on the torus. For a collection $C$ of pairwise disjoint normal curves on the torus containing at least one non-trivial component, the slope of $C$ is the isotopy class of a non-trivial component. Two slopes are \textbf{complementary} if their Haken sum is a collection of trivial curves. 

Jaco and Sedgwick show that if two normal surfaces are compatible and meet $\partial M$ in non-trivial slopes, then these slopes are either equal or complementary~\cite[Proposition 3.7]{Jaco-decision-2003}. 
This allows the possibility that some or all boundary curves of a normal surface are trivial.
The projective solution space of normal surface theory is denoted by $\mathcal{P}(\tri).$  Given a normal surface $S$, its \textbf{carrier} is the unique minimal face $\mathcal{C}(S)\subset \mathcal{P}(\tri)$ that contains the projectivised normal coordinate of $S.$ 

\begin{proposition}[Jaco-Sedgwick, Corollary 3.8~\cite{Jaco-decision-2003}]
\label{cor:JS-key}
Let $M$ be an orientable, compact, connected 3--manifold with $\partial M$ a single torus. 
Suppose $\tri$ is a triangulation of $M$ that restricts to a one-vertex triangulation of $\partial M.$ Suppose $S$ is a normal surface and $\partial S \neq 0.$ Assume also that $\partial S$ contains at least one essential curve. There are at most two slopes (complementary ones) for all surfaces in the carrier $\mathcal{C}(S)\subset \mathcal{P}(\tri).$
\end{proposition}

\begin{lemma}
\label{lem:taut_fundamental}
Let $M$ be an orientable, compact, irreducible 3--manifold with $\partial M$ a single, incompressible torus. 
Let $\tri$ be a 0--efficient triangulation of $M$ and let $\alpha\in \text{im}(\pi_1(\partial M)\to \pi_1(M))$ be an even slope. If there is an incompressible and $\partial$--incompressible surface $S$ in $M$ with $[\partial S] = \alpha$ and $\chi(S) = - || \; \alpha \; ||,$ then there is a fundamental surface $F$ of $\tri$ with $[\partial F] = \alpha$ and $\chi(F) = - || \; \alpha \; ||.$
\end{lemma}

\begin{proof}
Let $S$ be an incompressible and $\partial$--incompressible surface $S$ with $[\partial S] = \alpha$ and $\chi(S) = - || \; \alpha \; ||.$ Since $M$ is irreducible and has incompressible boundary, we can isotope $S$ to be a normal surface. If $S$ is not fundamental, then $S$ is a sum of fundamental surfaces. Now $S$ has boundary a single curve. It follows from \Cref{cor:JS-key} that there is only a single summand, $F$, with non-empty boundary and all other summands are closed normal surfaces. To see this note that otherwise $S$ would have disconnected boundary or a trivial curve in the boundary. Whence $[\partial F] = [\partial S] = \alpha.$ Since $\chi(S) = - || \; \alpha \; ||,$ we have $\chi(F) \le \chi(S).$ Since Euler characteristic is additive under Haken sums and the triangulation is $0$--efficient, this forces $\chi(F) = \chi(S).$
\end{proof}

\begin{corollary}
Let $M$ be an orientable, compact, irreducible 3--manifold with $\partial M$ a single, incompressible torus. 
Let $\tri$ be a 0--efficient triangulation of $M$. Then $\alpha$ is a minimising slope for $M$ if and only if there is a 
fundamental surface $F$ of $\tri$ with $[\partial F] = \alpha$ and $\chi(F) = - || \; M \; ||.$
\end{corollary}

\begin{proof}
Suppose $\alpha$ is a minimising slope for $M.$ Then there is a surface $S$ in $M$ with $[\partial S] = \alpha$ and $\chi(S) = - || \; M \; ||.$ Since $|| \; M \; ||$ is minimal, there is no surface in $M$ with a single boundary component and larger Euler characteristics than $S.$ Hence $S$ is incompressible and $\partial$--incompressible. The result therefore follows from \Cref{lem:taut_fundamental}.
\end{proof}

The above corollary gives an algorithm to compute $ || \; M \; ||$ and the set of all minimising slopes. This can be improved to compute the norm of every even slope:

\begin{theorem}
\label{thm:slope norm-algo}
Let $M$ be an orientable, compact, irreducible 3--manifold with $\partial M$ a single, incompressible torus.
Suppose $\tri$ is a $0$-efficient triangulation of $M$, and $(\meridian_2,\longitude_2)$ a $2$-torsion framing of $\partial M.$

There is an algorithm that, upon input $\tri$ with a $2$-torsion framing and an even slope, computes the norm of this slope.
\end{theorem}

\begin{proof}
Let $\delta$ be an even slope and $S$ be a taut surface for $\delta.$ Hence $S$ has a single boundary curve of slope $\delta$, is incompressible, and satisfies $|| \; \delta \; || = -\chi(S).$ If $S$ is not $\partial$--incompressible, we can perform boundary compressions on $S$ until we have an incompressible and $\partial$--incompressible surface. Denote the resulting sequence of surfaces $S = S_0,$ $S_1,$ \ldots, $S_n$ with boundary slopes $\delta = \delta_0, \delta_1, \ldots, \delta_n$, where $S_n$ is an incompressible and $\partial$--incompressible surface, and $S_i$ is obtained from $S_{i-1}$ by a single boundary compression. In particular, $\chi(S_i) = \chi(S_{i-1})+1.$

Let $\delta_i$ be the slope of $S_i.$ Since $S=S_0$ is a taut surface for $\delta_0,$ it follows inductively that $S_i$ is a taut surface for $\delta_i,$ since otherwise reversing the process of boundary compressions by adding saddles to a taut surface would result in a surface of higher Euler characteristic. It now follows from \Cref{lem:taut_fundamental} that $S_n$ is isotopic to a fundamental surface with respect to the $0$-efficient triangulation $\tri.$

A priori, there are infinitely many possibilities for the slope $\delta_1$ that result from a boundary compression on $S_0$, as can be seen from the Farey tree described in \Cref{fig:Farey}. However, the facts that we have a sequence of boundary compressions terminating at a fundamental surface (and hence at one of only a finite list of slopes), and that the set of boundary slopes can be organised via the trivalent tree that is the dual 1--skeleton of the Farey tesselation, makes the problem of determining $\chi(S)$ finite.

In order to compute the norm of the given even slope $\delta,$ we first compute the (finite) set of all fundamental surfaces. This is equivalent to enumerating a Hilbert basis on the projective solution space.
Amongst these, we select the surfaces with connected boundary. Denote these surfaces $F_i$ and their slopes $\alpha_i.$ We remark that $F_i$ may not be taut for $\alpha_i.$ For all the surfaces in the list that have the same slope, we only keep one of maximal Euler characteristic.

Consider the Farey tesselation associated with the framing $(\meridian_2,\longitude_2).$ Recall that the dual 1--skeleton is an infinite trivalent tree. We use this to define distances between ideal triangles (equivalently, isotopy classes of 1--vertex triangulations). 
For every slope, there are infinitely many ideal triangles with a vertex at this slope.
 We make a canonical (but arbitrary) choice by choosing, for each $\alpha_i$ the ideal triangle $\tau(\alpha_i) = (\alpha_i, \beta_i, \gamma_i)$ that is at shortest distance from the base triangle $\big( \frac{1}{0}, \frac{0}{1}, \frac{-1}{1} \big).$ Note that this is characterised by the relationship $\alpha_i = \beta_i \oplus \gamma_i$ in Farey addition (whereby numerators and denominators are simply added). If $\alpha_i$ is the ideal vertex of some ideal triangle $\tau_i,$ then we call $\alpha_i$ the \emph{even slope} of $\tau_i.$ Note that the set of all ideal triangles with the same even slope corresponds to an infinite line in the dual tree.

The effect of a boundary compression on the slope of a surface is exactly one step in the dual tree between triangles with distinct even slopes. The reverse step is the addition of a saddle. Hence the sequence $S = S_0,$ $S_1,$ \ldots, $S_n$ described above corresponds to a path without backtracking between $\tau(\delta_0)$ and $\tau(\delta_n)$ in the dual tree. The difference $\chi(S_n) - \chi(S_0)$ is the number of edges in this path that have endpoints in triangles of different even slopes. This difference is well-defined since the dual $1$--skeleton of the Farey tesselation is a tree.

Given even slopes $\alpha$ and $\beta$, let $d(\alpha, \beta)$ be the number of edges in the shortest path between $\tau(\alpha)$ and $\tau(\beta)$ that have endpoints in triangles of different even slopes. Note that this is independent of the choice of $\tau(\alpha)$ and $\tau(\beta)$ as triangles with those even slopes.
It follows that 
\[ || \; \delta \; || = -\chi(S) = \min_{F_i} \{ - \chi(F_i) + d(\alpha_i, \delta) \}\]
This completes the proof. 
\end{proof}

The proof leads to the following algorithm:
\begin{algorithm}
\label{algo:slopenorm}
Compute the norm of a given even slope
\begin{algorithmic}
  \State {\sc Input:}
  \State 1. $\tri$ $0$-efficient triangulation of $M$
  \State 2. $(\meridian_2,\longitude_2)$ $2$-torsion framing of $\partial M$
  \State 3. $p/q$ even boundary slope
  \State
  \State Compute fundamental surfaces of $\tri$
  \For{$S$ fundamental surface with connected boundary}
    \State Compute boundary slope of $S$ with respect to $(\meridian_2,\longitude_2)$
    \If{boundary slope already in Farey tesselation data structure}
    \State update norm = min (norm, $-\chi(S)$)
    \Else 
    \State insert boundary slope and norm into Farey tesselation data structure
    \EndIf
  \EndFor
  \State Insert $p/q$ into boundary slope into Farey tesselation data structure
  \State
  \State \Return minimum of (distance + norm) of $p/q$ to boundary slopes in the Farey tesselation data structure    
\end{algorithmic}
\end{algorithm}

\begin{remark}[(Farey tesselation data structure.)]
In the proof, we have chosen to associate a canonical triangle in the Farey tesselation with an even slope. In practice, one may use any triangle with the slope, thus saving some computations. If the norms of many slopes are to be computed, then it would be worthwhile to first setup a data structure containing the norm for each slope of a fundamental surface. 

It is important to note that we do not have to deal with the infinite object that is the Farey tesselation, but with a finite tree that is modelled on the dual 1-skeleton of the Farey tesselation. 
The nodes of our data structure are the ones of the dual 1--skeleton corresponding to even boundary slopes realised as boundary slopes of fundamental surfaces. Arcs are established between nodes along edges in the dual 1-skeleton of the Farey tesselation. 
This is an efficient procedure thanks to the Euclidean algorithm. It may add auxiliary nodes that are common to paths coming from different nodes.
The arcs are assigned weights equal to the number of edges in the dual 1-skeleton of the Farey tesselation that have endpoints in triangles of different even slopes.
With this setup, providing a query boundary slope amounts to inserting this extra slope into the data structure (exactly as before), and computing weighted path lengths without backtracking to all other nodes. 
\end{remark}


\section{Crosscap number of knots}
\label{sec:crosscap}

We are now restricting our view to knot exteriors with the following special property. Suppose $N$ is a closed, orientable 3--manifold, and $K\subset N$ a knot with $[K] = 0 \in H_1(N; \mathbb{Z}_2).$ 

Let $\nu(K)$ be an open regular neighbourhood of $K.$ We assume that the exterior $M = N\setminus \nu(K)$ is irreducible and contains no embedded non-separating torus and no embedded Klein bottle. For instance, this is the case for any knot in $N=\mathbb{S}^3$, and it is the case for any hyperbolic knot in N, i.e. when $M$ has a complete hyperbolic structure of finite volume.

The knot $K$ is \textbf{trivial} in $N$ if there is a properly embedded disc in $M$ that has non-trivial slope on $\partial M.$ In particular, $K$ is non-trivial if and only if $M$ has incompressible boundary.

The \textbf{crosscap number} $\cross(K)$ of a non-trivial knot $K\subset N$ is defined by:

\[ \cross(K) = \min \{ \; 1 - \chi(S) \; \mid S \text{ is a non-orientable spanning surface for } K \; \}\]

The crosscap number of a trivial knot is defined to be zero. The subtle difference between crosscap number and slope norm is that the crosscap number is not simply obtained by computing the minimal norm over all slopes of spanning surfaces, since it also takes into account the orientability class of a taut surface.


\subsection{Geometric framings and spanning surfaces}

The closure of $\nu(K)$ is a solid torus, and the curve $\meridian_g$ on $\partial M$ bounding the meridian disc for this solid torus is \textbf{the geometric meridian} for $K.$ This information allows us to pass between $(N,K)$ and $(M, \meridian_g).$ 

A \textbf{spanning surface} for $K$ in $N$ is an embedded, connected surface $S$ in $N$ with $\partial S = K.$
If $S$ is a spanning surface for $K$ in $N$, then $F = S \cap M$ has the property that $\partial F$ has algebraic intersection number $\pm 1$ with $\meridian_g$ on the boundary torus (after choosing orientations for both curves). Conversely, suppose $F$ is a properly embedded surface in $M$ with a single boundary component. If $\partial F$ has algebraic intersection number $\pm 1$ with $\meridian_g$, then $F$ extends to a spanning surface of $K$ in $N$.

The condition that $[K] = 0 \in H_1(N; \mathbb{Z}_2)$  implies that $K$ has a (possibly non-orientable) spanning surface. The intersection pairing with the meridian shows that $\meridian_g$ maps to a non-trivial element of $\text{im} ( H_1(\partial M, \mathbb{Z}) \to H_1(M, \mathbb{Z}_2)),$ and hence is a 2--meridian.

Recall the definition of a homological framing in \Cref{def:Homological_framing}. In the setting here, it is more natural to define a geometric framing that includes the class of the boundary of the meridian disc as a generator.

\textbf{Geometric framing, I.} Suppose $K$ has an orientable spanning surface $S.$ Then the boundary of $S$ is the homological longitude of $M$ and, due to the intersection pairing, the geometric meridian is a homological meridian. In particular, $K$ has an orientable spanning surface if and only if $[K] = 0 \in H_1(N; \mathbb{Z})$. By the discussion above, in this case, every non-orientable spanning surface has boundary slope of the form $\meridian_g^{2k}  \longitude_\infty$ (recall that we represent isotopy classes of unoriented curves by choosing a representative with non-positive longitudinal coordinate).
Any orientable spanning surface can be turned into a non-orientable spanning surface by attaching a saddle, and iteratively adding saddles shows that the set of all slopes of spanning surfaces is precisely the set 
\[S(K) = \{ \meridian^{2k}_\infty  \longitude_\infty \mid k \in \mathbb{Z}\}\]
 for any homological meridian. Alternatively, one may define a \textbf{geometric longitude} $\longitude_g = \longitude_\infty$ and hence 
\[S(K)=\{ \meridian^{2k}_g  \longitude_g \mid k \in \mathbb{Z}\}\]
 is the set of all boundary slopes of spanning surfaces. See \Cref{fig:evenslopes} for the subtree of spanning slopes sitting inside the dual of the Farey tesselation.

\textbf{Geometric framing, II.} Now suppose $K$ has no orientable, but a non-orientable spanning surface $S.$ This is the case if and only if $[K] = 0 \in H_1(N; \mathbb{Z}_2)$ and $[K] \neq 0 \in H_1(N; \mathbb{Z}).$  The ``only if" direction follows from the existence of $S$ and the discussion above, and the ``if" direction from the construction given in the proof of \Cref{lem:Z2-hom-dual-surface}. By the above, we have $[\partial S] =  \ \meridian^{2k_0}_\infty \longitude^{q_0}_\infty,$ for some $q_0,k_0 \in \mathbb{Z}$, $q_0$ and $2k_0$ co-prime, and $q_0> 0.$ Since this is isotopic to the core curve of $\nu(K),$
$\meridian_g = \meridian^{r_0}_\infty \longitude^{s_0}_\infty,$ where $2k_0s_0 - q_0r_0 = \pm 1.$ 
Each slope of a non-orientable surface is zero in $H_1(M; \mathbb{Z}_2)$ and meets the geometric meridian once algebraically. Again, by adding saddles to $S$, one can obtain all possible slopes that arise this way, and 
we conclude that the slopes of all spanning surfaces are precisely the set 
\[S(K)=\{ \meridian_g^{2k} ( \meridian^{2k_0}_\infty \longitude^{q_0}_\infty) \mid k \in \mathbb{Z}  \}\]
 To obtain a more pleasing description, we define a \textbf{geometric longitude} $\longitude_g$ as follows.
Fix a Euclidean norm on $H_1(\partial M, \mathbb{Z})$ with the property that $|| \; \longitude_\infty \; || = 1$ and $|| \; \meridian_g \; || = 1.$
Then define $\longitude_g$ to be a shortest curve in $S(K).$ It then follows that 
\[S(K)=\{ \meridian^{2k}_g \longitude_g \mid k \in \mathbb{Z}\}\]

\textbf{Computing intersection numbers.} Suppose $\tri$ is a triangulation of $M$ with the property that the induced triangulation $\tri_\partial$ of $\partial M$ is a two-triangle triangulation of the torus. The purpose of this section is to show how to set up the equations to determine whether a given normal surface is a spanning surface.

Given a normal surface $S$ with non-empty boundary, we obtain $\partial S = c$ as a (not necessarily connected) normal curve in $\tri_\partial$.  The normal curve $c$ is connected if and only if the greatest common divisor of its coordinates equals one. Hence $S$ can only be a spanning surface if this is the case. 

We assume that the geometric meridian $\meridian_g$ is given as a normal curve with respect to $\tri_\partial$. We now explain how to compute the minimal number of intersections between the isotopy classes of any two essential, connected normal curves (such as $\meridian_g$ and $\partial S = c$). This  makes use of the fact that, on the torus, the geometric and the algebraic intersection numbers of oriented curves coincide.

\begin{figure}[htb]
  \centerline{\includegraphics[width=13cm]{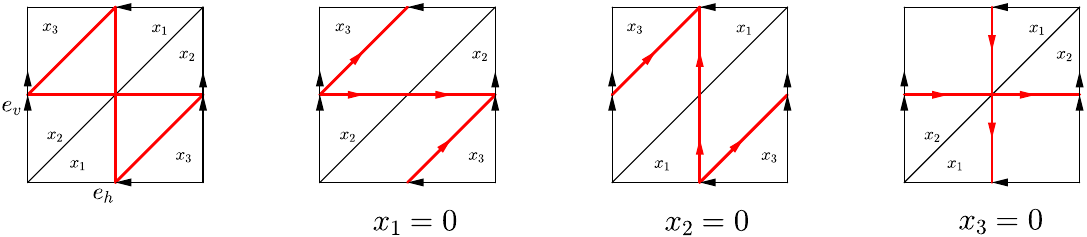}}
  \caption{The figure-8 graphs and their orientations for the three types of essential normal curves.
  \label{fig:normalcurves}}
\end{figure}

Represent the triangulation $\tri_\partial$ as the identification space of a square with a diagonal as shown in \Cref{fig:normalcurves}, and label the normal coordinates as shown. It was shown in \cite{Jaco-decision-2003} that the normal coordinates $(x_1, x_2, x_3)$ of one triangle determine the normal coordinates in the second triangle as indicated in the figure.
A connected essential normal curve contains no vertex linking curves, and hence its normal coordinate satisfies $x_i=0$ for at least one $i \in \{ 1,2,3\}$. Again, we check that $\partial S = c$ satisfies this requirement.

It follows that an essential normal curve can be isotoped to be in the neighbourhood of a figure--8 graph on the torus that is composed of four normal arcs. There are three such graphs, as shown in \Cref{fig:normalcurves}.
The normal coordinates can be viewed as weights on the edges of this graph. 
It was observed in \cite{bachman16complicatedHeegaardSplittings} that each essential normal curve can be given a canonical orientation which depends only on which normal coordinate is zero. This, in turn can be viewed as an orientation of the edges of the corresponding figure--8 graph as shown in \Cref{fig:normalcurves}.

In particular, the oriented intersection numbers between an oriented normal curve and two oriented edges $e_v$ and $e_h$ of the triangulation $\tri_\partial$ can be read off from the normal coordinate. We call them \textbf{oriented edge weights}. Hence the intersection number of two essential normal curves can be computed as a determinant in these oriented edge weights. With respect to the labelling shown in \Cref{fig:normalcurves}, the normal coordinate $x = (x_1, x_2, x_3)$ results in the oriented edge weights 
$(x_2+x_3, x_3)$ if $x_1=0$, 
$(x_3, x_1+x_3)$ if $x_2=0$, and
$(x_2, -x_1)$ if $x_3=0$.
The intersection number between two oriented essential normal curves curves is then the determinant of the $2\times 2$ matrix with these oriented edge weights as columns. 
For example, if 
$x = (0, x_2, x_3)$ and $y = (y_1, 0, y_3),$ then the intersection number is $(x_2+x_3)(y_1+y_3) - x_3y_3 = x_2y_1+x_2y_3+x_3y_1.$

The up-shot of this discussion is that given a normal curve with respect to $\tri_\partial$ that represents the meridian, we have an algorithm to determine whether a normal surface is a spanning surface.

\textbf{Efficient suitable triangulations.} 
One of the tools used in \cite{burton12-crosscap} is the following concept. We say that $\tri$ is an \textbf{efficient suitable triangulation} of $M$ if there is no normal 2--sphere with respect to $\tri,$ and the induced triangulation $\tri_\partial$ of the boundary has a single vertex and the geometric meridian is isotopic with a boundary edge.
As in \cite{burton12-crosscap}, it follows easily from \cite[Proposition 5.15 and Theorem 5.20]{jaco03-0-efficiency} that if $K$ is a non-trivial knot in $N$, then $M$ has an efficient suitable triangulation and that this can be constructed algorithmically from any triangulation of $M.$ Namely, one first constructs a 0--efficient triangulation of $M.$ If one of the edges in the boundary is isotopic with the geometric meridian, then we are done. Otherwise, we layer tetrahedra on the triangulation until one of the edges is the geometric meridian. A minimal layering sequence can be detemined from the Farey tesselation. Now a 0--efficient triangulation has no normal 2--spheres. Consider a step in the layering procedure. If the triangulation before layering a tetrahedron on the boundary has no normal 2--sphere, then so does the triangulation after layering a tetrahedron, since any normal surface that meets the new tetrahedron in a disc has non-empty boundary. So it follows inductively that an efficient suitable triangulation can be obtained from a 0--efficient triangulation.

\textbf{Normalisation.} An excellent discussion of the procedure that constructs normal surfaces from properly embedded surfaces transverse to a triangulation can be found in Matveev's book~\cite[\S3.3.3]{Matveev-algorithmic-2007}. See also \cite{burton12-crosscap} for a similar discussion to what follows. Normalisation moves are either isotopies, removal of trivial components, or compressions along circles of intersection of the surface with triangles of the triangulation, or boundary compressions along discs that have part of their boundary in the interior of boundary edges of the triangulation.

Suppose $S$ is a properly embedded surface in $M$ that is incompressible and transverse to the triangulation. Since $M$ is irreducible, any normalisation move that would result in a 2--sphere or properly embedded disc split off from $S$ can be avoided by a suitable isotopy of $S$ that is supported in a regular neighbourhood of a ball bounded by the 2--sphere or co-bounded by the disc. In particular, this isotopy removes some intersection points of $S$ with the edges or some intersection circles of $S$ with the triangles (possibly both). Instead of a normalisation move of this type, we perform the associated isotopy.
If $S$ can be transformed to a normal surface in this way, then we say that $S$ \textbf{normalises by isotopies}. 

The only normalisation moves that cannot be replaced by an isotopy are boundary compressions with the following property. The boundary compression disc $D$ is contained in a 3--simplex $\Delta$ of the triangulation and $\partial D = \gamma \cup \gamma'$ with the property that $\gamma$ is an arc contained in the interior of an edge $e\subset \partial M$ of $\Delta$ and $\gamma'$ is an arc contained in $S$. Moreover, $D \setminus \gamma$ is contained in the interior of $\Delta$. Note that such a boundary compression decreases $S\cap e$ by two points and leaves the intersections with all other boundary edges unchanged. The normalisation procedure may involve multiple boundary compressions. The up-shot of this discussion is that if $S$ meets one boundary edge in a single point, then this is also true of any surface obtained from $S$ by applying the normalisation procedure (with or without the modification of using isotopies).


\subsection{Crosscap number via suitable triangulations}

\Cref{thm:crosscap-suitable}, which is proved in this section, is a result about the crosscap number of certain knots in closed 3--manifolds, which, in particular, gives an algorithm to compute the crosscap number of an arbitrary knot in $\mathbb{S}^3.$ The result requires the use of efficient suitable triangulations, as defined in the previous section. \Cref{thm:crosscap} in \Cref{sec:efficient} is the equivalent result for arbitrary $0$-efficient triangulations. Both results share \Cref{lem:main_lemma}, which is why forward references to \Cref{thm:crosscap} and \Cref{sec:efficient} appear in this section.

The proof of \Cref{thm:crosscap-suitable} shows that the following algorithm to compute crosscap number is correct:

\begin{algorithm}
\label{algo:crosscapno-std}
Compute crosscap number of knot $K$ in $3$-manifold $M$, satisfying conditions of \Cref{thm:crosscap-suitable}.
\begin{algorithmic}
  \State {\sc Input:}
  \State 1. $\tri$ efficient suitable triangulation of $M$ with boundary $\tri_\partial$
  \State 2. $\meridian$ meridian of $K$ represented by an edge of $\tri_\partial$
  \State
  \State Compute $S_0$, the set of fundamental surfaces of $\tri$
  \State Compute $S_1 = \{S\in S_0 \mid \partial S \cap \meridian = 1 \}$, the set of spanning fundamental surfaces
  \State Compute $x = \max \{\chi (S) \mid S \in S_1\}$, maximum Euler characteristic of all spanning fundamental surfaces
  \If{$x=1$}
    \State \Return $\cross(K)=0$
  \EndIf
  \If{$x = 0$}
    \State \Return $\cross(K)=1$
  \EndIf
  \State Compute $S_2 = \{ S \in S_1 \mid \chi(S) = x\}$, the set of maximal Euler characteristic spanning fundamental surfaces
  \If{$S_2$ contains non-orientable surface}
    \State \Return $\cross(K) = 1-x$
  \Else
    \State \Return $\cross(K) = 2-x$
  \EndIf
\end{algorithmic}
\end{algorithm}

\begin{proof}[Proof of \Cref{thm:crosscap-suitable}]
Suppose $S_o$ is an orientable spanning surface of maximal Euler characteristic, and $S_n$ is a non-orientable spanning surface of maximal Euler characteristic. Then $c(K) = \min (\ 1-\chi(S_n), 2-\chi(S_o) \ ).$ If there is no orientable spanning surface, then $c(K) = 1-\chi(S_n).$
The definitions imply that $c(K) \le \min (A, B),$ since any orientable spanning surface can be turned into a non-orientable spanning surface by adding a saddle appropriately.

Both $S_o$ and $S_n$ are incompressible due to the maximality condition on Euler characteristic.

Since $\partial M$ is a torus and $S_o$ is orientable, this also implies that $S_o$ is $\partial$--incompressible. We may therefore assume that $S_o$ is normal in $M.$ By \cite[Corollary 3.8]{Jaco-decision-2003}, two compatible normal surfaces with non-empty boundary either have the same slope (hence their sum has at least two boundary curves) or complementary boundary curves (hence their sum has boundary containing a trivial curve). Hence only one of the fundamental surface summands, $G_o$, yielding $S_o$ has non-empty boundary and $\partial S_o = \partial G_o.$
Since Euler characteristic is additive and there are no normal 2--spheres, we have $\chi(S_o) \le \chi(G_o)$. Since $\partial S_o = \partial G_o$, $G_o$ is also a (not necessarily orientable) spanning surface.

If $\chi(S_n)<\chi(S_o),$ then $G_o$ must be orientable.\footnote{See the proof of \Cref{thm:knot-genus} for an argument that $G_o$ is always orientable if $S_o$ is of least weight in its isotopy class.} Also, $\chi(S_n)<\chi(S_o)$ implies $\cross(K) = 2-\chi(S_o) \le 1-\chi(S_n).$ Since $G_o$ is an orientable spanning surface, we have $\chi(G_o)=\chi(S_o),$ and therefore $\cross(K) = 2-\chi(G_o) \ge B.$
Thus, $c(K)=B.$\footnote{Note that in this case, we either have $c(K) = B <A$ or $c(K) = B =A$. See \Cref{sec:computation} for examples of both cases.}

Hence assume $\chi(S_n)\ge \chi(S_o).$ In this case $\cross(K) = 1-\chi(S_n) < 2-\chi(S_o).$

Now $S_n$ may not be $\partial$-incompressible. We use the following argument from \cite{burton12-crosscap}.
Since $\partial S_n$ has algebraic intersection number one with the geometric meridian $\meridian_g$, we may isotope $S_n$ in $M$ so that it meets the edge in $\tri_\partial$ that represents $\meridian_g$ in exactly one point. We now isotope $S_n$ so that it is transverse to the triangulation. This still meets $\meridian_g$ in exactly one point. As noted above, any surface obtained by applying the normalisation procedure to $S_n$ results in a surface meeting $\meridian_g$ in exactly one point, and hence a spanning surface.

In particular, any non-trivial boundary compression involved in putting $S_n$ into normal form results in a spanning surface of larger Euler characteristic. Since $\chi(S_n)\ge \chi(S_o)$, this surface must again be non-orientable, which contradicts the maximality of the Euler characteristic of $S_n.$ 
Hence $S_n$ can be normalised by isotopies. 
As above, we see that only one of the fundamental surface summands, $G_n$, yielding $S_n$ has non-empty boundary. Since Euler characteristic is additive, we have $\chi(S_n) \le \chi(G_n)$. Since $\partial S_n = \partial G_n$, $G_n$ is also a spanning surface.

We now have the following cases:

If $\chi(S_n)>\chi(S_o),$ then $G_n$ must be non-orientable. This forces $\chi(G_n)=\chi(S_n),$ and we have $c(K)\ge A,$ which implies $c(K) =  A<B.$

If $\chi(S_n) = \chi(S_o)$, then $\chi(G_o) = \chi(S_n) = \chi(S_o) = \chi(G_n).$
The proof is continued with \Cref{lem:main_lemma}, where it is shown that there is at least one non-orientable fundamental spanning surface with this maximal Euler characteristic. Hence $c(K) = A<B.$ 
\end{proof}

\begin{lemma}
\label{lem:main_lemma}
Suppose $M$ and $\tri$ are as in the hypothesis of \Cref{thm:crosscap-suitable} or \Cref{thm:crosscap}. Also assume, in the case of \Cref{thm:crosscap}, that $\cross(K)<Z.$

Suppose $S_o$ is an orientable spanning surface of maximal Euler characteristic, and $S_n$ is a non-orientable spanning surface of maximal Euler characteristic. If $\chi(S_n) = \chi(S_o)$, then there is at least one non-orientable fundamental spanning surface $F$ with $\chi(F)=\chi(S_n) = \chi(S_o).$
\end{lemma}

\begin{proof}
We prove this by contradiction. 
We know that every non-orientable spanning surface of maximal Euler characteristic is isotopic to a normal surface 
(since we either assume that the triangulation is suitable or that $\cross(K)<Z$). 
We also know (from the last paragraph in the proof of each \Cref{thm:crosscap-suitable,thm:crosscap})
that there is at least one fundamental spanning surface of Euler characteristic $\chi(S_n) = \chi(S_o).$
Suppose every such fundamental spanning surface is orientable, and hence $S_n$ is isotopic to a non-trivial Haken sum $S$ of normal surfaces. We then show that this implies that $S$ must be orientable as well, a contradiction.

Here is the outline of our proof:
   \begin{enumerate}
    \item {\bf Setup:} Let  $S$ be a non-orientable normal spanning surface with $\chi(S) = \chi(S_n)$ that is of least-weight.
    \item {\bf Setting up the Haken sum for $S$:} We show that $S = R+T$, where $T$ is a torus, $R$ connected (incompressible), and no patch is a disc.
    \item {\bf $R$ meets $T$ an even number of times:} We show that $R\cap T$ is an even number of essential curves on $T$.
    \item {\bf All components of $R \setminus T$ above $T$ are annuli:} We show that any patch of $R$ in the component of $M\setminus T$ that does not contain $\partial M$ must be a (separating) annulus.
    \item {\bf $R+T$ is orientable:} We use the above statements to conclude that $R+T$ must be orientable.
  \end{enumerate}

  \paragraph{Proof setup:}
Suppose $x= \chi(S_n)=\chi(S_o)$ is the maximum Euler characteristic of a fundamental spanning surface.  
We assume that all fundamental spanning surfaces with maximum Euler characteristic are orientable. Since $K$ is non-trivial, none of these is a disc. Since there are no orientable spanning surfaces of Euler characteristic $0$, we can hence assume for the remainder of the proof that all orientable fundamental spanning surfaces have strictly negative Euler characteristic.
  
We know that a non-orientable \emph{normal} spanning surface $S$ with $\chi(S)=x$ exist. Furthermore, we may assume that $S$ has least weight amongst all non-orientable normal spanning surfaces with Euler characteristic $x.$
  
By hypothesis, $S$ is not fundamental.
  
\paragraph{Setting up the Haken sum for $S$:}
By \cite[Corollary 3.8]{Jaco-decision-2003} (see \Cref{cor:JS-key}), two compatible normal surfaces with non-empty boundary either have the same slope (hence their sum has at least two boundary curves) or complementary boundary curves (hence their sum has boundary containing a trivial curve). Hence only one of the fundamental surface summands, $F$, yielding $S$ has non-empty boundary and $\partial F = \partial S.$ In pacticular, $F$ is a spanning surface for $K$ and has negative Euler characteristic.
  
By hypothesis, $\tri$ does not admit any normal spheres, and since $M$ is irreducible with non-empty boundary a torus, it does not admit any embedded $\mathbb{R}P^2$. Hence, every surface in the Haken sum giving $S$ has non-positive Euler characteristic. Since $F$ is a spanning surface and Euler characteristic is additive in Haken sums, the maximality of Euler characteristic $x$ amongst all spanning surfaces forces $\chi(F)=x$, and all other summands have Euler characteristic zero. Since $F$ is a fundamental spanning surface and $\chi(F)=x$, it follows that $F$ is orientable. Since $M$ contains no Klein bottles, all other summands are fundamental tori.
  
  \medskip
  
Note that removing any normal torus from any Haken sum with result $S$ must produce an orientable surface, since $S$ is least-weight. 
  
Let $R$ be such an orientable surface, and let $T$ be the missing torus such that $R+T = S$. There are potentially many such decompositions and we claim that for at least one of them $R$ is connected: 
Assume that we have $R+T = S$ with $R$ disconnected. Hence, $R$ must consist of one spanning surface of maximal Euler characteristic, and a number of tori. Since $R+T=S$ is connected, all of these extra torus components of $R$ must intersect $T$. Pick one of these tori, denote it by $T'$, and use the remaining components of $R$ to form the Haken sum $R' = (R\setminus T') + T$. By construction, $S = R' + T'$ and $R' \cap T' \subset R \cap T$ is a proper subset. Iterating this process eventually yields a decomposition $S=R+T$ with $R$ connected.
  
  \medskip  
  
  Since $R$ is an orientable spanning surface of maximal Euler characteristics, it is incompressible and boundary-incompressible.
   
  \medskip
  
As is customary when talking about \textbf{Haken sums}, we refer to the connected components of $(R \cup T )  \setminus \nu(R \cap T)$ as \textbf{patches}. Every patch is a compact orientable subsurface of $S$ with non-empty boundary. Since both $R$ and $T$ are orientable, the patches are connected on $S$ by annuli contained on the frontier of $\nu(R \cap T)$. These are called the \textbf{exchange annuli}. The core curve of an exchange annulus is called a \textbf{trace curve}. Note that a trace curve corresponds to the boundary curves of patches on both $R$ and $T$ that are joined by the corresponding annulus.

The complementary annuli on the frontier of $\nu(R \cap T)$ are the \textbf{irregular annuli}. Let $\gamma \subset R \cap T.$ Attaching the exchange annuli to the patches is called a \textbf{regular exchange} at $\gamma$, and attaching the irregular annuli to the patches is called an \textbf{irregular exchange}.

Amongst all ways to write $S=R+T$, where $R$ is an orientable normal spanning surface of maximal Euler characteristic and $T$ is a normal torus, we assume we have chosen one with the least number of patches. Note that $|R\cap T|\neq \emptyset$ as $S$ is non-orientable.

\paragraph{No patch of $S = R+T$ is a disc.}
This follows almost verbatim from the proof of Lemma 2.1 of~\cite{Jaco-algorithm-1984}. We repeat a slightly adapted version of the beginning of the proof here, as our set-up is slightly different. However, the end-game remains the same.

Suppose $\gamma \subset R\cap T$ and denote the two associated trace curves $\gamma'$ and $\gamma''.$ 
Suppose $\gamma'$ bounds a patch $D'$ that is a disc on $S.$ Since a disc is 2-sided in $M$, there is an embedded disc $D\subset M$ with interior disjoint from $S$ and boundary curve $\gamma''.$ Since $S$ is incompressible this implies that $\gamma''$ also bounds a disc $D''$ on $S.$ Since $D'$ is a patch, we have $D'' \nsubseteq D'.$ We claim that $D''$ is not a patch.

If $D' \subset D'',$ then $D''$ is clearly not a patch. 

Hence assume that $D' \nsubseteq D''.$  Then $D' \cap D'' = \emptyset.$ Perform an irregular exchange along $\gamma$ and regular exchanges along all other curves in $R+T.$ If an irregular annulus does not join $D'$ and $D''$, then we can construct a 2-sphere in $M$ that separates $S$, a contradiction since $S$ is connected. Hence the union of $D'$, $D''$ and an irregular annulus is an embedded 2--sphere in $M.$ Now if both $D'$ and $D''$ are patches, then we may perform a regular exchange along only $\gamma$ and no other component of $R\cap T$ to obtain new normal surfaces $R'$ and $T'$ with $S = R' + T'$. Moreover, $R'$ is isotopic with $R$ and $T'$ is isotopic with $T.$ So $R' + T'$ still satisfies our hypothesis on the Haken sum that $R'$ is an orientable normal spanning surface of maximal Euler characteristic and $T'$ is a normal torus. Now $R' + T'$ has two patches fewer than $R+T,$ since $D'$ and $D''$ are now contained in patches. It follows that $D''$ is not a patch.

So in either case, we establish that $D''$ is not a patch. To conclude, we know that for each patch $D'$ that is a disc, there is an associated disc $D''$ on $S$ which is not a patch and with the property that the boundary curves of $D'$ and $D''$ are associated with the same curve in $R \cap T.$
Now the construction and argument given in the last four paragraphs in the proof of Lemma 2.1 of~\cite{Jaco-algorithm-1984} results in a surface of lower weight than $S$ and isotopic with $S$. Hence no patch is a disc, and in particular, no patch is of positive Euler characteristic. 
  
  \paragraph{Irregular exchanges.}
  Recall that we have established that $R \cap T$ is a collection of essential curves on $T$ slicing up $T$ into a set of annuli. Any such situation can be fully reconstructed from the schematic picture shown in \Cref{fig:intersection}, where the horizontal line represents the torus $T$, and the vertical lines represent $R$ slicing through $T$. Every possible Haken sum is then described by an orientation at each intersection for how the Haken sum operates, and the transverse orientation for every sheet of $R$.
  
    \begin{figure}[h]
    \centerline{\includegraphics[width=\textwidth]{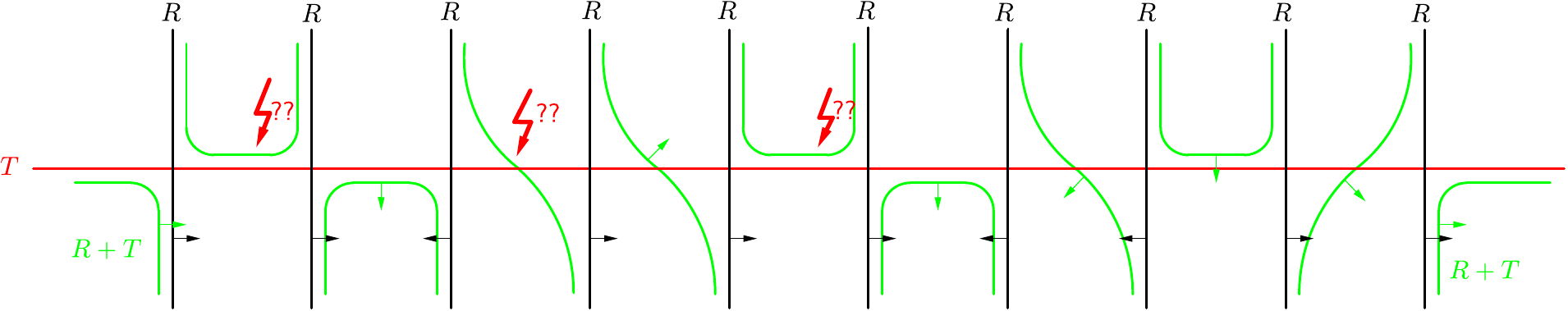}}
    \caption{Cross-section of the intersection of $T$ and $R$ together with transverse orientation of $R$ and marked potential obstructions to orientability of the resolution realised by the Haken sum $S=R+T$. \label{fig:intersection}}
  \end{figure}    

Suppose any non-empty subset of components of $R\cap T$ are resolved via irregular exchanges, whilst all other components are resolved via regular exchanges. This results in a surface that is not normal, has non-empty boundary and is not necessarily connected. However, it contains a component $C$ that is a maximum Euler characteristic spanning surface. Since we have performed at least one irregular exchange, we know that $C$ is not a normal surface and hence normalises to a normal surface of lower weight than $S.$ 

We claim that $C$ is orientable. This is the only place where we have different arguments given the hypotheses of the theorems:

First assume the hypotheses of \Cref{thm:crosscap-suitable}. Since $\partial C = \partial S$, the surface $C$ is a maximum Euler characteristic spanning surface that meets the meridian edge in exactly one point and hence normalises by isotopies. In particular, the resulting normal surface is again a spanning surface. The minimality of the weight of $S$ implies that this spanning surface (and therefore $C$) is orientable.
  
Next assume the hypotheses of \Cref{thm:crosscap} and that $\cross(K)<Z.$ Hence if $C$ is non-orientable, then it normalises by isotopies. But this contradicts the minimality of $S$. Hence $C$ is orientable.

This proves the claim. It follows that if we perform at least one irregular exchange, then there is a component $C$ that is a maximum Euler characteristic spanning surface and that normalised to an orientable spanning surface of maximal Euler characteristic. We will repeatedly make use of this observation.
  
  \paragraph{$R$ meets $T$ an even number of times:} 
  Since $R+T$ is non-orientable, and $R$ and $T$ are orientable, every orientation reversing loop in $R+T$ must pass through patches of both $R$ and $T$. Take one such loop $c \subset R+T$ that minimises the number of times it intersects the exchange annuli between patches of $R$ and $T$. We claim that $c$ intersects each exchange annulus in at most one arc from one boundary component to the other.
  
  To see this first note that if $c$ intersects an exchange annulus in an arc going back to the same  boundary component we can simply isotope it out of the exchange annulus. A contradiction to the assumption that $c$ minimises intersections with exchange annuli.
  
  If $c$ intersects the same exchange annulus in more than one arc, take two such arcs that are next to each other and isotope them into a small disc containing a section of the exchange annulus and a piece of $R$ and $T$ on either side. Delete the two arcs meeting the disc and connect the four ends outside the exchange annulus to obtain a curve $c'$ meeting the exchange annuli of $R+T$ fewer times than $c$ (see \Cref{fig:disk}).
    
   We show that one component of $c'$ must still be orientation reversing. Referring to \Cref{fig:disk} we start at end $1$ which can be connected to end $2$, $3$, or $4$.
  \begin{itemize}
    \item If it is connected to end $2$, we immediately obtain a contradiction to the assumption that $c$ is connected. See \Cref{fig:disk} first row.
    \item If it is connected to end $3$, then end $2$ must connect to end $4$. Tracing transverse orientations through both arcs and connecting them in the disc either leaves us with no, one, or two orientation reversing components of $c'$. We can now check that if none or both components were orientation reversing, then $c$ would have been orientation preserving. A contradiction. Hence, exactly one of them is orientation reversing. A contradiction to the assumption that $c$ intersect the exchange annuli of $R+T$ a minimum number of times. See \Cref{fig:disk} middle row.
    \item If it is connected to end $4$, then end $2$ must be connected to end $3$. Similarly to the case before, we can trace orientations through $c'$ thereby checking all four choices. As before, in each case we obtain contradictions to either $c$ being orientation reversing or $c$ having minimal intersection with the exchange annuli of $R+T$. See \Cref{fig:disk}, bottom row.
  \end{itemize}
  
   \begin{figure}
     {\includegraphics[height=5.3cm]{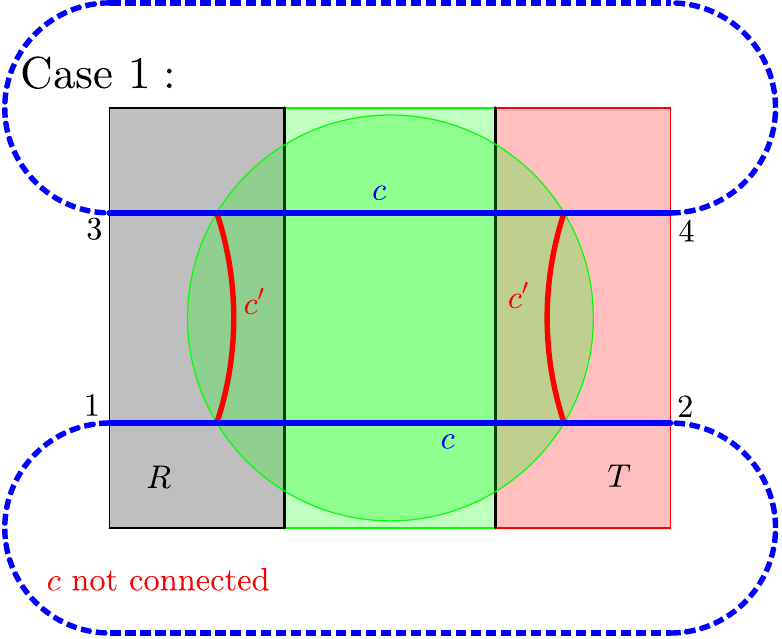}}
     
     \bigskip     
     
     {\includegraphics[height=4.6cm]{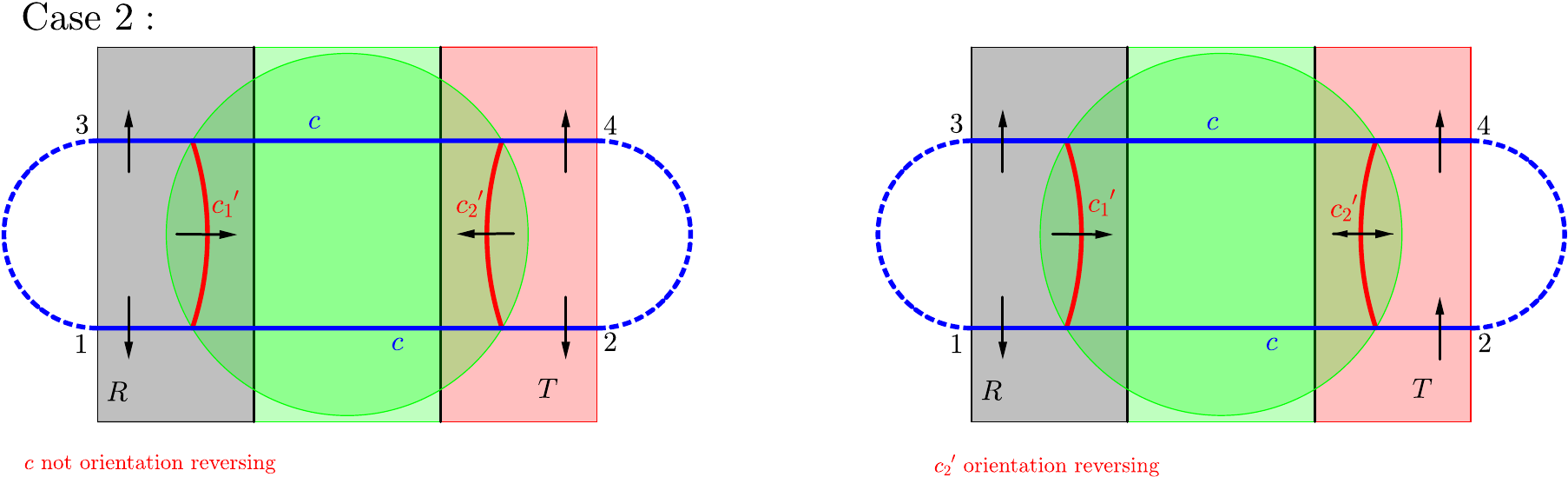}}

     \bigskip     
          
     {\includegraphics[height=5.5cm]{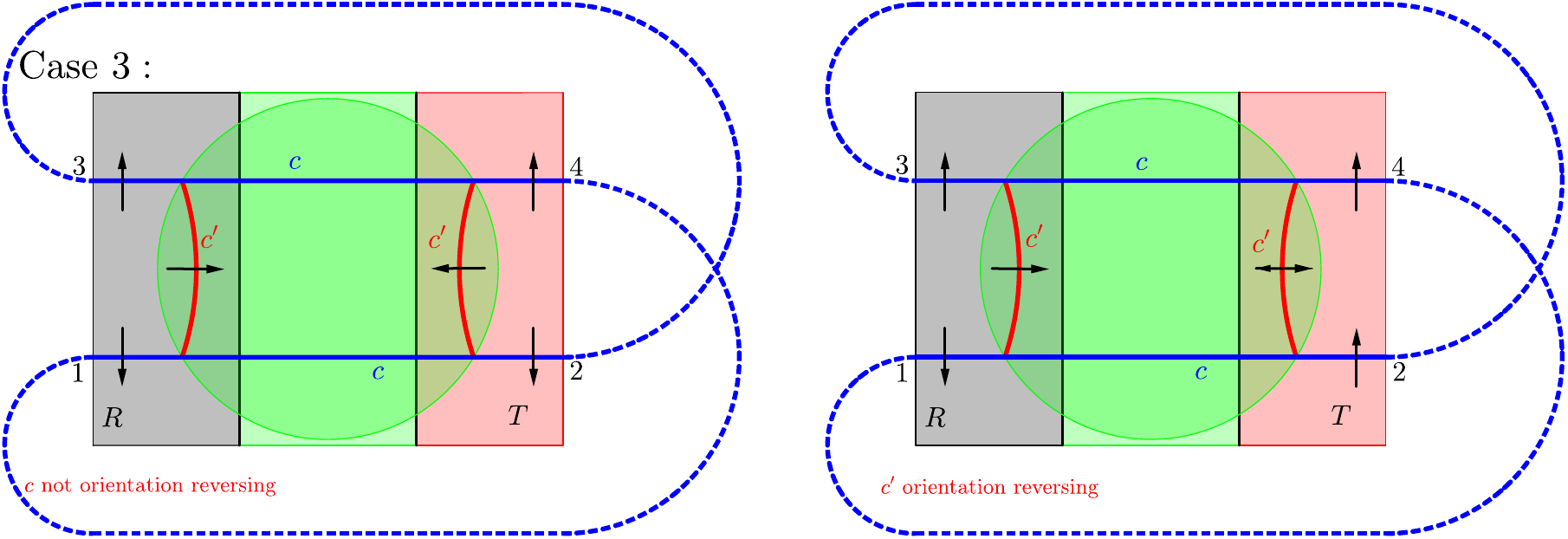}}
   
     \caption{Case 1: $1$ is connected to $2$. $c$ is not connected. Case 2: $1$ is connected to $3$. Two cases of $16$ for choosing orientations on $c$ near $1$, $2$, $3$, and $4$ are shown. The others follow by flipping $R$ and $T$, top and bottom, and direction of all arrows. Case 3: $1$ is connected to $4$. Again, only two out of $16$ cases are shown and the others follow by symmetry. \label{fig:disk}}
   \end{figure}
  
  Altogether it follows that $c$ intersects every exchange annulus at most once.
  
  Now, assume that $c$ is disjoint to the two exchange annuli coming from some component of $R\cap T$. If this is the case, we can make an irregular exchange along that component and regular exchanges along all other components of $R\cap T$. Then the resulting new surface still contains an orientation reversing loop $c$ and hence is non-orientable -- a contradiction to $S$ being least-weight as explained above.
    
  Next, assume that $c$ intersects both exchange annuli coming from some component of $R\cap T$. Consider the solid torus $D^2 \times S^1$ that is a regular neighbourhood of this component and containing the two exchange annuli. Choose an open regular neighbourhood of the two exchange annuli on $S$. There is an isotopy of c on $S$ that is the identity on the complement of this neighbourhood, and has the effect that the intersection of $c$ with these exchange annuli is contained in some disc $D = D^2 \times \{x\}$. In $D$ we now have a similar picture as before and we refer to \Cref{fig:disk}. Performing an annular exchange yielding a new surface $(R+T)'$ cuts $c$ open along two small arcs with ends $1$, $2$, $3$, and $4$. Connecting ends $1$ and $3$ and ends $2$ and $4$ through the newly added annuli yields another curve $c'$ in the changed surface. An analysis of how orientations can be traced on $c'$ results in an orientation reversing loop assuming that $c$ was orientation reversing. It follows that $(R+T)'$ must be non-orientable. A contradiction to the assumption that $R+T$ is least-weight.
  
  Altogether we conclude that $c$ must meet exactly one of the two exchange annuli coming from a component of $R\cap T$ for every such component. Running through $c$ we meet patches of $R$ or $T$ and exchange annuli in an alternating fashion. Moreover, a patch of $R$, followed by an exchange annulus, must then be followed by a patch of $T$ and vice versa (i.e., $c$ runs through patches of $R$ and $T$ in an alternating fashion as well). In particular, $c$ must run through an even number of exchange annuli. But since the number of exchange annuli meeting $c$ is in bijection to the components of $R\cap T$ we conclude that the number of components of $R\cap T$ must be even.
  
\paragraph{All patches of $S$ on one side of $T$ are annuli:}
The Euler characteristic of $S$ equals the sum of the Euler characteristics of all patches on $S$. Since no patch is a disc, the Euler characteristic of any patch is non-positive. We already know that all patches of $S$ contained on $T$ are annuli. Since $T$ is separating in $M$ we say that the component of $M\setminus T$ containing $\partial M$ is \emph{below} $T$ and that the other component is \emph{above} $T$. 

Since $|R\cap T|$ is even, we now perform (possibly irregular) exchanges on $R \cap T$ such that we obtain a new surface $S'$ that is isotopic to a surface disjoint from $T$. This is achieved by alternating the exchanges such that patches on $R$ above $T$ are joined via annuli on $T$ with patches above $T$ and patches below $T$ with patches below $T$. This results in one surface, $R'$, below $T$ and one surface, $T'$, above $T$. As above, $\chi(R') = \chi(S)$ and this forces all patches that are not contained on $R'$ to be annuli.

  \paragraph{$R+T$ is orientable:}
Every properly embedded annulus above $T$ is separating in $M\setminus T$ since otherwise we have a non-separating torus in $M$. Take an innermost annulus patch of $S$, i.e. a patch contained in a component $C$ of $R \setminus T$ with the property that the frontier of $C$ bounds an annulus on $T$ that does not contain any components of $R \cap T.$
There are four possibilities for how the Haken sum $R+T$ connects the patch on $C$ with patches on $T$. One of them produces a separate connected torus component and hence contradicts the connectedness (least-weight) assumption for $R+T$ (see \Cref{fig:annulus} on the top right). Two solutions give us the opportunity to resolve the Haken sum in the opposite way on one of the crossings, producing an extra torus component and not changing the non-orientability of the rest of the Haken sum (see \Cref{fig:annulus} bottom row). This leaves us with the last option which is an exchange of annuli between the summands (see \Cref{fig:annulus} on the top left). 
  
  \begin{figure}[h]
    \centerline{\includegraphics[width=.65\textwidth]{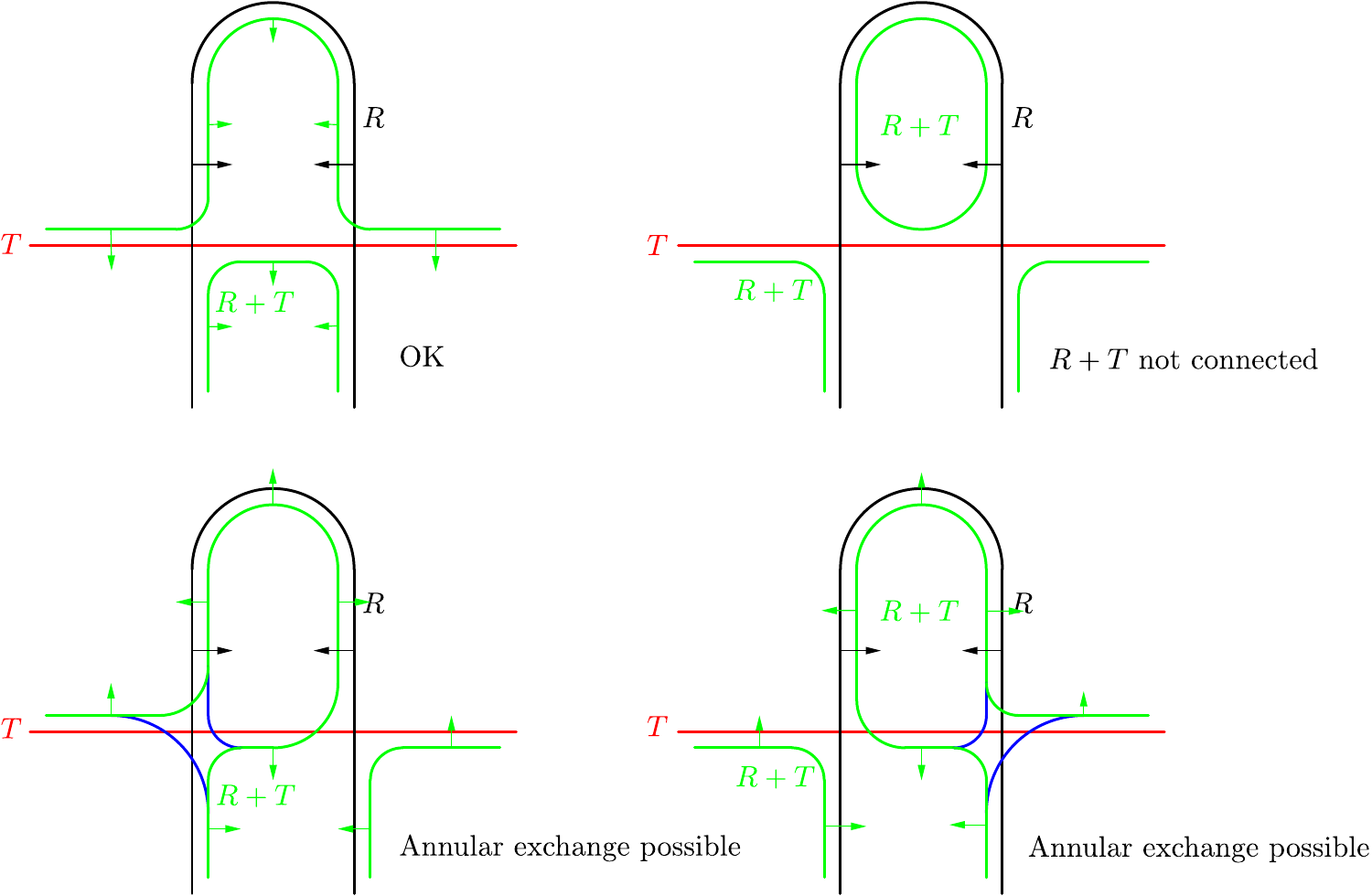}}
    \caption{Resolving a separating annulus in all four possible ways. \label{fig:annulus}}
  \end{figure}  
  
  Given that all components of $R \setminus T$ on the side of $M$ not containing $\partial M$ are separating annuli, we can iterate this argument stating that the entire Haken sum can be resolved this way. But then $R+T$ is disconnected. A contradiction.
  
We conclude that there is no non-orientable normal spanning surface $S$ with $\chi(S)=\chi(S_o)$. This contradicts our hypothesis that $\chi(S_o) = \chi(S_n)$ and the hypothesis that $S_n$ is normal. 
\end{proof} 


\subsection{Crosscap number via arbitrary 0--efficient triangulations}
\label{sec:efficient}

A trade-off in the previous algorithm is that one cannot apply it to arbitrary $0$--efficient or minimal triangulations.

\textbf{Even integral subtree.} 
With respect to the geometric framing $(\meridian_g, \longitude_g),$ construct the Farey tesselation. We claim that the dual 1--skeleton restricted to all ideal triangles that are labelled with the slopes of spanning surfaces is connected. 
Note that these are precisely the ideal triangles with labels of the form $\frac{2k}{1},$ where $k\in \mathbb{Z}.$  
We call this the \textbf{even integral subtree}, and denote it $\mathcal{F}_e.$ This tree is shown in \Cref{fig:evenslopes}.

\begin{figure}
  \centerline{\includegraphics[width=0.5\textwidth]{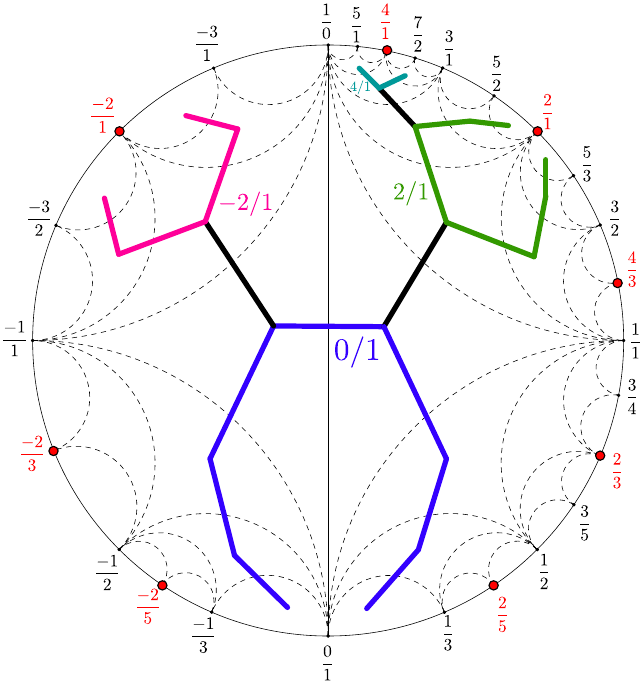}}
  \caption{The even integral subtree $\mathcal{F}_e$ representing spanning slopes in the dual of the Farey tesselation. Every spanning slope is represented by an infinite path of a given colour, paths are connected by a single saddle attachement (black arcs); shown is only a small portion of $\mathcal{F}_e$.\label{fig:evenslopes}}
\end{figure}

First suppose $k>0.$ There is a Farey triangle with vertices $1/0, 2k/1, (2k-1)/1.$
This is since we can act by the element 
$\begin{pmatrix} 1 & 2k \\ 0 & 1 \end{pmatrix}$
on the base triangle $1/0, 0/1, -1/1.$ Then flipping across the ideal edge $[1/0, (2k-1)/1]$ gives
the triangle with vertices $1/0, (2k-1)/1, (2k-2)/1.$ Hence inductively we arrive at the base triangle with even slope 
$0/1.$

Now suppose $k<0.$ Then there is a Farey triangle with vertices $1/0, 2k/1,
(2k+1)/1.$ This is obtained by acting by 
$\begin{pmatrix} 1 & 2k \\ 0 & 1 \end{pmatrix}$
on the triangle $1/0, 0/1, 1/1.$ Again, flipping across the ideal edge $[1/0, (2k+1)/1]$ gives a
triangle with even slope $(2k+2)/1.$ So inductively we arrive at the ideal triangle with vertices $1/0, 0/1, -1/1,$ which shares an edge with the base triangle. This completes the proof that the even integral subtree is connected.

For any slope $\frac{p}{q},$ we define its \textbf{even integral subtree distance}, $d(\frac{p}{q}, \mathcal{F}_e)$ to be the number of edges in the shortest path between $\tau(\frac{p}{q})$ and $\mathcal{F}_e$ that have endpoints in triangles of different slopes.

\begin{proof}[Proof of \Cref{thm:crosscap}]
It follows from the definitions that $\cross(K) \le A$ and $\cross(K) \le B.$ Note that if $S$ is a fundamental non-spanning surface for  $K$ with connected essential boundary, then adding $d(\partial S, \mathcal{F}_e)$ saddles according to the corresponding shortest path in the Farey tesselation gives a non-orientable spanning surface
of Euler characteristic $\chi(S) - d(\partial S, \mathcal{F}_e)$ for $K.$ Hence $\cross(K) \le Z.$ So $\cross(K) \le \min (A, B, Z),$ and we need to show equality.

Let $S_o$ be an orientable spanning surface of maximal Euler characteristic, and let $S_n$ be a non-orientable spanning surface of maximal Euler characteristic. We have $\cross(K) = \min (\ 1-\chi(S_n), \; 2-\chi(S_o) \ ).$

As in the proof of \Cref{thm:crosscap-suitable}, $S_o$ is incompressible and $\partial$--incompressible, and hence may be normalised using isotopies. As before, there is a fundamental surface $G_o$ with $\chi(S_o) \le \chi(G_o)$ and $\partial S_o = \partial G_o$. Hence $G_o$ is also a spanning surface. If $\chi(S_n)<\chi(S_o),$ then $G_o$ must be orientable. This forces $\chi(G_o)=\chi(S_o),$ and we have $c(K)=B.$

Hence assume $\chi(S_n)\ge \chi(S_o).$ In this case $\cross(K) = 1-\chi(S_n) < 2-\chi(S_o)\le B$.

Note that $S_n$ is incompressible, but $S_n$ may not be $\partial$-incompressible. 

Normalisation of $S_n$ may involve a finite number of non-trivial boundary compressions resulting in surfaces that are not spanning surfaces. Each non-trivial boundary compression increases Euler characteristic by one. Suppose $S_n$ normalises to the normal surface $S'_n$, and, topologically, the latter is obtained from the former by performing $k> 0$ non-trivial boundary compressions. Then $\chi(S_n) = \chi(S'_n)-k.$ Now if $S'_n$ is not fundamental, then we obtain a fundamental surface $G_n$ with $\partial G_n = \partial S'_n$ and $\chi(S'_n) \le \chi(G_n).$ The maximality of $\chi(S_n)$ implies that $\chi(S'_n) = \chi(G_n)$ and $k = d(\partial G_n, \mathcal{F}_e)$ since any surface obained by adding saddles to $G_n$ is non-orientable (regardless of whether $G_n$ is orientable or not). Hence 
\[ \cross(K) = 1-\chi(S_n) = 1 - \chi(G_n) + d(\partial G_n, \mathcal{F}_e) \ge Z\]
Hence $\cross(K) = Z.$ 

We may therefore assume that $\cross(K) < \min ( B, Z).$
So $\chi(S_n)\ge \chi(S_o)$ and \emph{every} non-orientable fundamental spanning surface with maximal Euler characteristic $\chi(S_n)$ normalises by isotopies. As above, we obtain a fundamental surface $G_n$ with $\partial G_n = \partial S_n$ and $\chi(S_n) \le \chi(G_n).$

If $\chi(S_n)>\chi(S_o),$ then $G_n$ must be non-orientable. This forces $\chi(G_n)=\chi(S_n),$ and we have $c(K)\ge A,$ which implies $c(K) =  A.$

If $\chi(S_n) = \chi(S_o)$, then $\chi(G_o) = \chi(S_n) = \chi(S_o) = \chi(G_n).$
The proof is continued with \Cref{lem:main_lemma}, where it is shown that there is at least one non-orientable fundamental spanning surface with this maximal Euler characteristic. Hence $\cross(K) = A.$ 
\end{proof}

\begin{remark}
It is known from work by Clark \cite{Clark78CrosscapNumber} that a knot in the 3--sphere has crosscap number zero or one if and only if it is a $(2,2k+1)$-cable of a knot ($k\in \mathbb{Z}$). Hence, one corollary of \Cref{thm:crosscap} is that a given $0$-efficient triangulation of a non-trivial knot complement in the 3--sphere is that of a $(2,2k+1)$-cable of a knot if and only if one of the fundamental surfaces is a M\"obius strip.
\end{remark}


\section{Genus of knots}
\label{sec:knot genus}

We wish to point out that in the setting of this paper, it is not difficult to recover a special case of a more general result of Schubert~\cite{Schubert1961-bestimmung} (which was originally proved in the context of normal surfaces with respect to handle decompositions). Namely, there is an algorithm to determine the genus of a knot using normal surface theory.

\begin{proof}[Proof of \Cref{thm:knot-genus}]
Suppose $S_o$ is an orientable spanning surface of maximal Euler characteristic. Since $\partial M$ is a torus and $S_o$ is orientable, this also implies that $S_o$ is $\partial$--incompressible. We may therefore assume that $S_o$ is normal in $M.$ Amongst all maximal Euler characteristic orientable normal spanning surfaces, we choose a surface $S$ of least weight.

By \cite[Corollary 3.8]{Jaco-decision-2003}, two compatible normal surfaces with non-empty boundary either have the same slope (hence their sum has at least two boundary curves) or complementary boundary curves (hence their sum has boundary containing a trivial curve). Hence only one of the fundamental surface summands, $F$, yielding $S$ has non-empty boundary and $\partial S = \partial F.$

Since Euler characteristic is additive and there are no normal 2--spheres, we have $\chi(S) \le \chi(F)$. 
We also note that the weight of $F$ is stictly less than the weight of $S$ unless $S=F$.

These two observations imply that if $F$ is orientable, then we have $S=F$ and hence $S$ is fundamental.

Hence suppose that $F$ is non-orientable. In this case, $S=F+G$ is a non-trivial Haken sum with $G\neq \emptyset$ a closed normal surface. We give all patches of $F+G$ the induced orientation from $S.$ Since $F$ is non-orientable, there is a curve $\gamma$ in $F \cap G$ where the induced orientations from $S$ on the two patches on $F$ meeting in $\gamma$ do not agree. Since the Haken sum is orientable, it is also the case that the induced orientations on $G$ do not agree. It follows that if one performs an irregular exchange at $\gamma$ and regular exchanges at all other intersection curves, then one obtains an orientable spanning surface with the same Euler characteristic as $S$ but which is not normal. Hence a normalisation of this surface will have lower weight than $S.$ This is a contradiction. Hence $F$ is indeed orientable.
\end{proof}


\section{Quadrilateral space}
\label{sec:quad_space}

We now provide some results that allow us to obtain minimising slopes and crosscap numbers of knots using computations in quadrilateral space. 

We begin with some general observations that will then be adapted under varying hypotheses. We assume that $M$ is an orientable, compact, irreducible 3--manifold with $\partial M$ a single, incompressible torus, and that $\tri$ is a $0$-efficient triangulation of $M$.

For a normal surface $F$, denote $[F]_Q$ the normal $Q$--coordinate of $F.$ If $F$ is not a vertex linking disc, then $[F]_Q\neq 0.$ If $[F]_Q\neq 0$ is not fundamental, then we write $[F]_Q = \sum [F_i]_Q$, where the $[F_i]_Q$ are fundamental normal $Q$--coordinates (with possible repetitions), and each of the corresponding normal surfaces $F_i$ is connected and not a vertex link. Such $F_i$ is called a  \textbf{$\mathbf{Q}$--fundamental surface.} With respect to standard coordinates, $\{ F_i\}$ is a compatible set of normal surfaces since triangle coordinates do not affect compatibility, and each $F_i$ is a fundamental normal surface that is not a vertex linking disc.

Let $D$ be a vertex linking disc. 
Then $F + kD = \sum F_i$ as a Haken sum of normal surfaces. 
The boundary of $F_i$ may consist of essential curves and trivial curves; or only consist of trivial curves; or be empty.
Since the $F_i$ are compatible normal surfaces, \cite[Corollary 3.8]{Jaco-decision-2003} implies that essential curves of at most two different slopes appear in the Haken sum, namely the slope of $\partial F$ and its complementary slope (which depends on the triangulation of the boundary).

Since the triangulation is $0$--efficient, we have $\chi (F_i) \le 0.$ Let $\{ F_i\} = \{ G_j\} \cup \{ H_n\}$ be a partition into two non-empty sets. Then there are normal surfaces $G$ and $H$ with the property that none of their connected components is a vertex linking disc, and integers $k'$ and $k''$ such that 
$G + k'D = \sum G_j$ and $H + k''D = \sum H_n.$
Now 
\[ 
F + kD = \sum F_i = \sum G_j + \sum H_n = G + H + (k'+k'') D
\]
since vertex linking discs can be isotoped to be disjoint from a Haken sum. The same reason implies $k \ge k' + k''.$ Note that
\[ 
\chi(G) = \chi(F) + (k - k' - k'') - \chi(H) \ge \chi(F)
\]
since $\chi(H) = \sum \chi(H_n) \le 0.$ We write $k = k_1 + k_0,$ where $k_0$ is the total number of trivial boundary components of the $F_i.$ 
As in \cite{Jaco-decision-2003}, let $\mu({\partial F})$ be the maximal normal arc coordinate of the slope $\partial F.$
Note that $k_1$ equals $\mu({\partial F})$ times the total number of essential boundary components of complementary slope in the $F_i.$ In particular, either $k_1=0$ or $k_1 \ge \mu({\partial F}).$

Now suppose $\partial F$ is a single essential boundary curve. Since $\partial F + k\partial D$ consists of $k+1= (k_1 +1) + k_0$ curves, \cite[Corollary 3.8]{Jaco-decision-2003} implies that the essential boundary curves of the surfaces $\{ F_i\}$ are $1+\frac{k_1}{\mu({\partial F})}$ connected curves of the slope of $F$ and $\frac{k_1}{\mu({\partial F})}$ connected curves of the complementary slope.

In particular, if $k=0$, then there is exactly one surface, say $F_1$, with $\partial F _1 \neq \emptyset.$ Hence we have $F_1$ fundamental, $\partial F_1 = \partial F$ a connected, essential curve, $\chi(F_1) \ge \chi(F)$ (by choosing $G=F_1$), and since there are no vertex linking discs in the sum, we see that $F_1$ has lower weight than $F$ unless $F=F_1.$ 
Also, $F_1$ has lower $Q$--weight than $F$ unless $F=F_1.$
Here, \textbf{weight} still refers to the number of intersections of a normal surface with the 1--skeleton, and \textbf{$Q$--weight} is the total number of quadrilateral discs.


\subsection{Minimising slopes} 

\begin{proposition}
\label{prop:Q-fundamental}
Let $M$ be an orientable, compact, irreducible 3--manifold with $\partial M$ a single, incompressible torus.
Suppose $\tri$ is a $0$-efficient triangulation of $M$.

Let $S$ be a connected surface of maximal Euler characteristic amongst all properly embedded surfaces in $M$ with boundary a single essential curve on $\partial M.$
Then there is a $Q$--fundamental surface $F$ with $\partial F = \partial S$ and $\chi(F) = \chi(S).$
\end{proposition}

\begin{proof}
Suppose $S$ is a surface of maximal Euler characteristic amongst all properly embedded surfaces with boundary a single essential curve  in $M.$ Then $S$ must be incompressible and $\partial$--incompressible, and hence normalises by isotopies. Amongst all normal surfaces with a single essential boundary component of the same slope as $S$ and the same Euler characteristic as $S$ choose one of least weight. Denote this surface $F$ (noting that $F$ may not be isotopic with $S$) and apply the preliminary observations. In particular, since $F$ is of least weight, either $[F]_Q$ is fundamental, or $k>0.$

Hence suppose $k>0.$ If $\frac{k_1}{\mu({\partial F})}$ is odd, let $\{ G_j \}$ be the subset of surfaces in $\{ F_i\}$ with slope complementary to $\partial F$. If $\frac{k_1}{\mu({\partial F})}$ is even (and hence $1+ \frac{k_1}{\mu({\partial F})}$ is odd), let $\{ G_j\}$ be the subset of surfaces in $\{ F_i\}$ with the same slope as $\partial F$. 
It follows that the surfaces in $\{ G_j\}$ all have the same slope. 
Since some surfaces may have trivial boundary components, we have
$G + k'_0 D = \sum G_j$ and $\partial G$ consists of an odd number of essential curves (and may have some trivial curves) and $k'_0$ is bounded above by the total number of trivial curves in the boundaries of the $G_j.$

In particular, we may attach annuli to pairs of essential boundary components of $G$ such that  (after a small isotopy) we obtain a properly embedded surface $G'$ with a single essential boundary component and $\chi(G') = \chi(G).$
Note that $G'$ may not be connected. Denote $G''$ the component of $G'$ with boundary containing the essential curve. 
Since the triangulation is 0--efficient and no component of $G'$ is a vertex linking disc, each component of $G'$ has non-positive Euler characteristic, and hence $\chi(G'') \ge \chi(G').$
If $G''$ has any boundary components that are trivial, we cap these off with disc, and denote the resulting surface again by $G''.$ This still satisfies $\chi(G'') \ge \chi(G').$

Similarly, $\sum H_n$ has boundary a family of parallel essential curves, and hence $H + k''_0 D = \sum H_n,$ where $k''_0$ is bounded above by the total number of trivial curves in the boundaries of the $H_n.$ We have $k'_0 + k''_0 \le k_0.$
This implies:
\[ 
F + (k_1 + k_0) D = F + kD = \sum F_i = \sum G_j + \sum H_n = G + H + (k'_0 + k''_0) D
\]
If $k_1 > 0,$ then $\chi(G'') \ge \chi(G') = \chi(G) = \chi(F) + (k_1 + k_0 - k'_0 + k''_0) - \chi(H) >\chi(F)= \chi(S),$ contradicting the maximality of the Euler characteristic of $S$ amongst all surface with boundary a single essential curve.

Hence $k_1=0.$ But then there is a unique surface in $\{ F_i\}$ with an essential curve in its boundary. Without loss of generality, assume this is $F_1.$ If $F_1$ only has one boundary component, then
$\chi(F_1) \ge \chi(F)$ and either the weight of $F_1$ is less than that of $F$ (which would be a contradiction) or  $F = F_1$ is $Q$--fundamental. If $F_1$ has more than one boundary component, then the other boundary components are trivial, and hence we may cap them off with discs, obtaining a surface $F'$ with $\partial F' = \partial F$ and $\chi(F') > \chi(F)=\chi(S)$, contradicting the maximality of the Euler characteristic of $S.$
Hence $[F]_Q$ is fundamental.
\end{proof}

We state the following immediate corollary:
\begin{corollary}
\label{cor:slope norm-Q}
Let $M$ be an orientable, compact, irreducible 3--manifold with $\partial M$ a single, incompressible torus. Suppose $\tri$ is a $0$-efficient triangulation of $M$. Then $\alpha$ is a minimising slope for $M$ if and only if there is a 
$Q$--fundamental surface $F$ of $\tri$ with $[\partial F] = \alpha$ and $\chi(F) = - || \; M \; ||.$
\end{corollary}

The above corollary gives an algorithm to compute $ || \; M \; ||$ and the set of all minimising slopes from the $Q$--fundamental solutions. However, in the presence of incompressible and $\partial$--incompressible surfaces at slopes other than the minimising slopes, we only obtain an upper bound on the norm of any slope if only the $Q$--fundamental solutions and not all fundamental solutions are computed.


\subsection{Crosscap number} 

This section gives a proof of \Cref{thm:Q-algo}. We organise the proof in three stages. It follows from the previous section that the crosscap number can be computed from the $Q$--fundamental solutions if a spanning slope is a minimising slope for $M.$ This is immediate in the case where a non-orientable surface achieves the minimising slope (\Cref{cor:Q-quick}), and requires a little more effort when all these surfaces are orientable (\Cref{cor:Q-procedure}). The proof is then completed by showing that we can always compute the crosscap number from the $Q$--fundamental solutions that are spanning surfaces.

\begin{corollary}
\label{cor:Q-quick}
Let $M$ be the exterior of a non-trivial knot $K$ in a closed 3--manifold $N$ with $[K] = 0 \in H_1(N; \mathbb{Z}_2).$ Suppose that $M$ is irreducible and contains no embedded non-separating torus and no embedded Klein bottle. Let $\tri$ be a $0$-efficient triangulation of $M$ and suppose that the coordinates for a meridian for $K$ on the induced triangulation $\tri_\partial$ of $\partial M$ are given. 

Suppose that amongst the $Q$--fundamental surfaces with a single boundary component, the maximal Euler characteristic is achieved by a non-orientable spanning surface $S$ for $K.$ Then $c(K) = 1-\chi(S).$
\end{corollary}

\begin{proof}
Suppose $S_o$ is an orientable spanning surface of maximal Euler characteristic, and $S_n$ is a non-orientable spanning surface of maximal Euler characteristic. Then $c(K) = \min (\ 1-\chi(S_n), 2-\chi(S_o) \ ).$ 

It follows from \Cref{prop:Q-fundamental} that there is no spanning surface in $M$ of larger Euler characteristic than the surface $S$ in the hypothesis. Hence $\chi(S_o) \le \chi(S) = \chi(S_n),$ and therefore $c(K) = 1 - \chi(S).$ 
\end{proof}

In order to break up the proof of our main theorem, we offer the following improvement to the previous corollary
in the context of a 0--efficient and suitable triangulation. This result is an auxiliary step towards our main result \Cref{thm:Q-algo}.

\begin{proposition}
\label{cor:Q-procedure}
Let $M$ be the exterior of a non-trivial knot $K$ in a closed 3--manifold $N$ with $[K] = 0 \in H_1(N; \mathbb{Z}_2).$ Suppose that $M$ is irreducible and contains no embedded non-separating torus and no embedded Klein bottle. Let $\tri$ be a $0$-efficient suitable triangulation of $M.$ 

Suppose that amongst the $Q$--fundamental surfaces with a single boundary component, the maximal Euler characteristic is achieved by a spanning surface $S$ for $K.$ 
Then $\cross(K)=\min (A', B'),$ where
\begin{itemize}
\item $A' = \min \{ \; 1 - \chi(S) \; \mid S \text{ is a non-orientable $Q$--fundamental spanning surface for } K \; \}$
\item $B' = \min \{ \; 2 - \chi(S) \; \mid S \text{ is an orientable $Q$--fundamental spanning surface for } K \; \}$
\end{itemize}
\end{proposition}

\begin{proof}
Suppose $S_o$ is an orientable spanning surface of maximal Euler characteristic, and $S_n$ is a non-orientable spanning surface of maximal Euler characteristic. The surface $S_o$ (if it exists) is isotopic to a normal surface. Since the triangulation is suitable, the same is true for $S_n.$ We may therefore assume that $S_n$ and $S_o$ are least weight normal representatives amongst all normal spanning surfaces in the same orientability class and with maximal Euler characteristic.

In the preliminary observation, let $F$ be a normal spanning surface of maximal Euler characteristic. First suppose $k_1>0.$ 
We let $\{ G_j\}$ be the subset of surfaces in $\{ F_i\}$ with boundary curves of the same slope as $F$, and $\{ H_n\}$ be the complementary set. As in the proof of \Cref{prop:Q-fundamental}, we write 
$G + k'_0 D = \sum G_j$, where $k'_0$ is bounded above by the total number of trivial curves in the boundaries of the $G_j$ and $G$ does not contain any vertex linking discs.
Hence $\partial G$ consists of $\frac{k_1}{\mu({\partial F})}+1$ essential curves and some trivial curves.

Similarly, $\sum H_n$ has boundary a family of $\frac{k_1}{\mu({\partial F})}$ parallel essential curves of complementary slope and some trivial curves, and we write $H + k''_0 D= \sum H_n,$ where no component of $H$ is a vertex linking disc and 
$k''_0$ is bounded above by the total number of trivial curves in the boundaries of the $H_n.$
This implies:
\[ 
F + (k_1 + k_0) D = \sum F_i = \sum G_j + \sum H_n = G + H + (k'_0 + k''_0) D
\]
Since $k_0\ge k'_0 + k''_0$ and we assume $k_1 >0,$ we have
 $\chi(G) = \chi(F) + (k_1 + k_0- k'_0 - k''_0) - \chi(H) > \chi(F),$ and similarly $\chi(H) >\chi(F).$ Now either $G$ or $H$ has an odd number of essential boundary components, and hence a connected component $X$ with an odd number of essential boundary components. By capping off trivial boundary components of $X$ by discs and connecting essential boundary components in pairs, we obtain a properly embedded surface $X'$ with a boundary a single essential simple closed curve, and $\chi(X') \ge \chi(X) > \chi(S).$ It follows that $X'$ is not a spanning surface since $F$ is a maximal Euler characteristic spanning surface. For future reference, we note that 
 \[
 \chi(X') \ge \chi(X) \ge \chi(H) \ge \chi(F) + (k_1 + k_0- k'_0 - k''_0) - \chi(G) \ge  \chi(F) + k_1 - \chi(G) \ge \chi(F) + \mu({\partial F}) > \chi(F)
 \]   
and $\partial X'$ is a single curve of complementary slope to $\partial F.$ We will analyse this in detail in the proof of \Cref{thm:Q-algo}.
 
Note that \Cref{prop:Q-fundamental} implies that there is a $Q$--fundamental surface $Y$ of larger Euler characteristic than $\chi(F)$ and with a single boundary curve (which is possibly not of complementary slope). The existence of $Y$ contradicts our hypothesis that amongst the $Q$--fundamental surfaces with a single boundary component, the maximal Euler characteristic is achieved by a spanning surface $S$ for $K.$ 

Hence $k_1=0$ and
\[ 
F + k_0 D  = \sum F_i
\]
We may assume that the boundary curves of the $F_i$ are pairwise disjoint since they are a single essential curve and a finite number of trivial curves.
There is exactly one surface, say $F_1$, with $\partial F \subseteq \partial F _1$. 
Now 
\[
\chi(F_1) = \chi(F) + k_0  - \sum_{i\ge 2} F_i \ge \chi(F) + k_0
\]
If $\partial F = \partial F _1,$ then $F_1$ is a spanning surface. Since $F$ is a spanning surface of maximal Euler characteristic, this implies $k_0=0$ and $\chi(F_1) = \chi(F).$
If $\partial F_1$ also contains trivial curves, then we may cap these off with discs to obtain a surface spanning surface $F'_1$ with $\chi(F'_1) > \chi(F_1) \ge \chi(F),$ which is a contradiction. 

Hence $k_0=0$ and we have $\partial F_1 = \partial F$ and $\chi(F_1) = \chi(F)$. Hence for every normal spanning surface $F$ of maximal Euler characteristic, we have 
\[ 
F  = \sum F_i
\]
where the $F_i$ are $Q$--fundamental and, without loss of generality, $F_1$ is a normal spanning surface of maximal Euler characteristic.  In particular, for each $i\ge 2,$ we have $\chi(F_i) = 0$ and $\partial F_i = \emptyset.$

If $\chi(S_n) > \chi(S_o),$ then $F_1$ is non-orientable and we have 
$c(K) = 1 - \chi(F_1) = A'< B'.$

Similarly, if $\chi(S_o) > \chi(S_n),$ then $F_1$ is orientable and we have $c(K) = 2 - \chi(F_1) = B' \le A'.$

Hence suppose $\chi(S_o) = \chi(S_n),$ and let $F = S_n.$ Again, if $F_1$ is non-orientable, then $c(K) = 1 - \chi(F_1)=A'.$ Hence suppose that amongst all $Q$--fundamental spanning surfaces, there is no non-orientable surface with Euler characteristic equal to $\chi(S_n).$ In particular, 
\[S_n = F_1 + \sum_{i\ge 2} F_i\]
where $F_1$ is orientable and $\sum_{i\ge 2} F_i$ is a closed surface of Euler characteristic zero, and hence a union of separating tori. We are therefore in the setting of the proof of Lemma~\ref{lem:main_lemma}. The arguments in the proof of that lemma only hinge on $S_n$ being of least weight and equalling a Haken sum of the form $F_1 + \sum_{i\ge 2} F_i,$ but make no use of the fact whether or not the $F_i$ are fundamental.
Hence we obtain a contradiction and there must be a non-orientable $Q$--fundamental spanning surface $F'$ with Euler characteristic equal to $\chi(S_n).$ Hence $c(K) = 1 - \chi(F')=A'$ and we are done.
\end{proof}

\begin{proof}[Proof of \Cref{thm:Q-algo}]
There is only one place in the proof of \Cref{cor:Q-procedure}, where the hypothesis was used that 
amongst the $Q$--fundamental surfaces with a single boundary component, the maximal Euler characteristic is achieved by a spanning surface $S$ for $K.$ 

Hence suppose $F$ is a spanning surface of maximal Euler characteristic, and that there is a non-spanning surface $X'$
with a single boundary curve of complementary slope to $\partial F$ and satisfying  
 \[
 \chi(X') \ge  \chi(F) + k_1 \ge \chi(F) + \mu({\partial F}) 
 \]   
Let $\gamma = \partial F$ and $\gamma^\perp = \partial X'.$ 
 Since $F$ is a spanning surface of maximal Euler characteristic, we have
 \[
 \chi(F) \ge \chi(X') - d(\gamma^\perp, \mathcal{F}_e) 
 \]
 Hence
 \[
  \chi(F) + d(\gamma^\perp, \mathcal{F}_e) \ge \chi(X')  \ge \chi(F) + \mu(\gamma)
 \]
 and so
 \begin{equation}
 \label{eq:weight-distance}
 d(\gamma^\perp, \mathcal{F}_e)  \ge \mu(\gamma)
 \end{equation}
 This reduces our proof to a calculation in the Farey tesselation, with the aim of obtaining a contradiction to the above inequality. The boundary slope of $F$ is $\gamma = \meridian^{2m}_g \longitude_g$ for some $m\in \mathbb{Z}.$ The complementary slope $\gamma^\perp$ depends on the boundary pattern of the triangulation of $\partial M.$ 
 
 \begin{figure}
  \centerline{\includegraphics[width=0.35\textwidth]{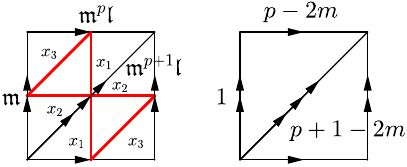}}
  \caption{The boundary pattern and arc coordinates; and signed edge weights of $\gamma$
    \label{fig:boundary_pattern}}
\end{figure}

There is $p\ge 0,$ such that the oriented boundary edges represent the classes $\meridian_g, \meridian_g^{p}\longitude_g,  \meridian_g^{p+1}\longitude_g.$ Hence the signed edge weights of $\gamma$ with the three edges are:
\[
\langle \meridian_g, \gamma \rangle = 1, \quad 
\langle \meridian_g^{p}\longitude_g, \gamma \rangle = p-2m, \quad
\langle \meridian_g^{p+1}\longitude_g, \gamma\rangle = p+1-2m
\]
This determines the signed edge weights of $\gamma$ with respect to the framing, as shown in \Cref{fig:boundary_pattern}.
Using the convention from \Cref{fig:normalcurves}, we can compute the 
 normal arc coordinate of $\gamma$ from this information. We then compute the normal arc coordinate of $\gamma^\perp,$ and hence the slope of $\gamma^\perp$ with respect to our framing. Since $\gamma^\perp$ is not a spanning slope, we show below that $\frac{p}{1} < \gamma^\perp < \frac{p+1}{1}$. Now exactly one of $\frac{p}{1}$ or $\frac{p+1}{1}$ is a spanning slope. This implies that we can compute $d(\gamma^\perp, \mathcal{F}_e)$ as the saddle distance of $\gamma^\perp$ to this spanning slope. Our proof is completed by showing that in each case, \Cref{eq:weight-distance} cannot be satisfied.
 
We remark that at this point, one could change the framing to simplify some of the notation, but we choose not to, as it does not simplify the argument.
 
\textbf{Case 1:} First suppose that $p-2m \ge 0.$ Then the normal arc coordinate of $\gamma$ is $(p-2m, 1, 0)$, and so $\mu(\gamma) = \max(1, p-2m).$ 

\begin{figure}
  \centerline{\includegraphics[width=\textwidth]{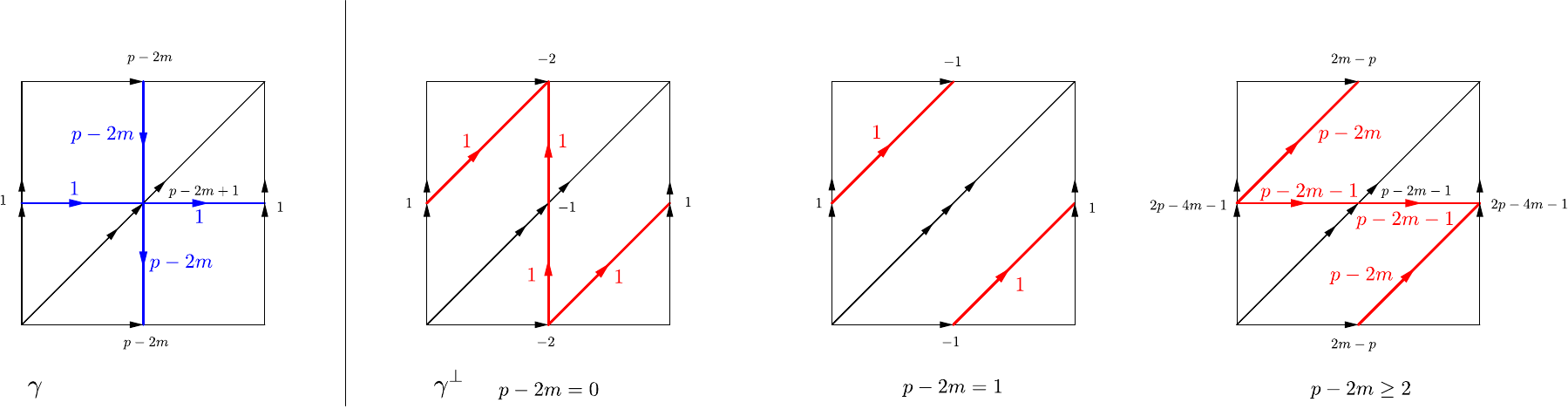}}
  \caption{Normal coordinate of $\gamma$ and its complementary curve in the first case
    \label{fig:case1complementary}}
\end{figure}

If $p-2m=0,$ then the complementary slope has normal arc coordinate $(1, 0, 1),$ and hence
$\gamma^\perp = \meridian_g^{p+2} \longitude_g$ is a spanning slope. This is a contradiction.

If $p-2m=1,$ then the complementary slope has normal arc coordinate $(0, 0, 1),$ and hence
$\gamma^\perp = \meridian^{p+1}_g \longitude_g$, which again contradicts $\gamma^\perp$ not being the slope of a spanning surface.

Hence $p-2m>1$ and so  $\mu(\gamma) = p-2m.$ Then the complementary slope has normal arc coordinate $(0, p-2m-1, p-2m),$ and hence satisfies
\[
\langle \meridian_g, \gamma^{\perp} \rangle = 2p-4m-1, \quad 
\langle \meridian_g^{p}\longitude_g, \gamma^{\perp} \rangle = 2m-p, \quad
\langle \meridian_g^{p+1}\longitude_g, \gamma^{\perp}\rangle = p-2m-1
\]
and so (switching to additive notation) we have:
\[
\gamma^\perp =  (2p^2-2m(1+2p)) \; \meridian_g \; + \; (2p-1-4m)\; \longitude_g
\]
To determine the distance to the even slope tree, we flesh out a part of the Farey tesselation. 
The continued fraction expansion determines a path of edges to $\gamma^\perp$ in the Farey tesselation. We compute:
\[
\frac{2p^2-2m(1+2p)}{2p-1-4m} = p+\cfrac{1}{2+\cfrac{1}{2m-p}} = p + [\; 2,\; 2m-p\;]
\]
Note that, since $\frac{p}{1} \oplus \frac{p+1}{1} = \frac{2p+1}{2}$ and
$\det \begin{pmatrix} p & 2p+1 \\  1 & 2 \end{pmatrix} = -1$, there is a triangle $\tau$ with vertices $\frac{p}{1}, \frac{2p+1}{2}$ and $\frac{p+1}{1}.$
Let
\[
\gamma_j = \frac{p+1}{1} \oplus j \; \frac{2p+1}{2} 
\]
where $j \in \{0, \ldots, j_0\}$ and $j_0 = p-2m-2.$
Note that 
\[
\frac{p+1}{1}= \gamma_0 > \gamma_1 > \ldots  > \gamma_{j_0} > \gamma^\perp > \frac{2p+1}{2} > \frac{p}{1}
\]
We have 
\[
\gamma_{j_0} = \frac{2p^2-(2m+1)(1+2p)}{2p-3-4m}.
\]
Since $\det \begin{pmatrix} 2p^2-(2m+1)(1+2p) & 2p+1 \\ 2p-3-4m & 2 \end{pmatrix} = 1$
and $\gamma_{j_0} \oplus \frac{2p+1}{2} = \gamma^\perp,$
there is a triangle $\tau'$ in the Farey triangulation with vertices $\gamma_{j_0}, \gamma^\perp$ and $\frac{2p+1}{2}.$ It follows from the expression for $\gamma_j,$ that there are $j_0 = p-2m-2$ triangles between the  triangles $\tau$ and $\tau'$ with pivot around the common vertex $\frac{2p+1}{2}.$ 

\begin{figure}
  \centerline{\includegraphics[width=0.45\textwidth]{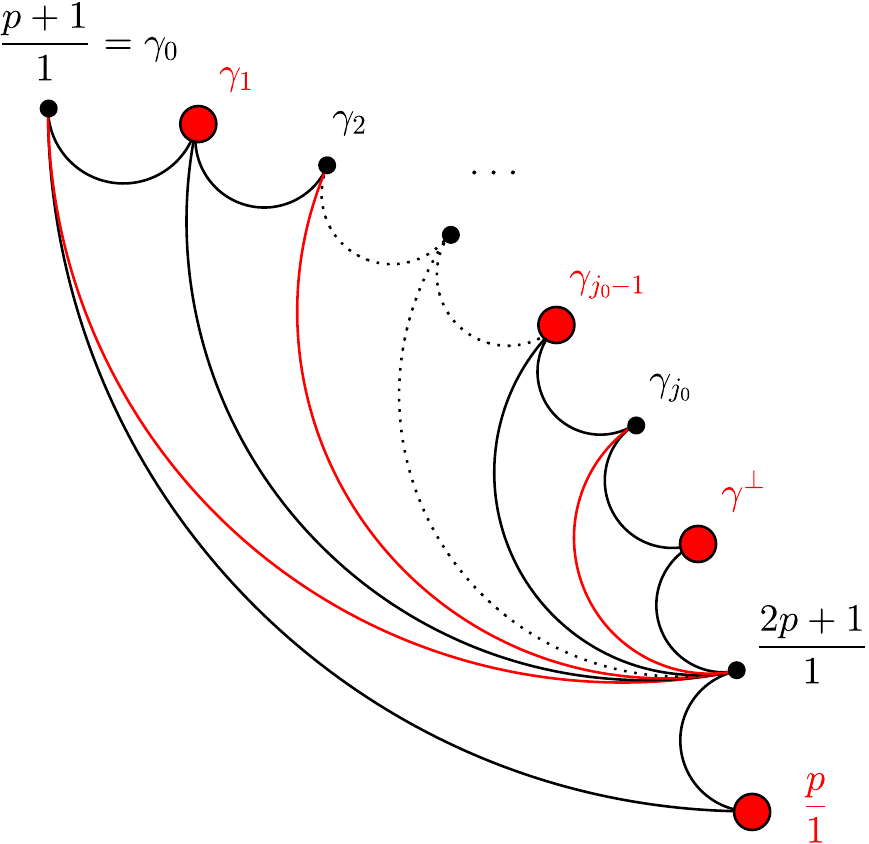} \,\,\, \includegraphics[width=0.45\textwidth]{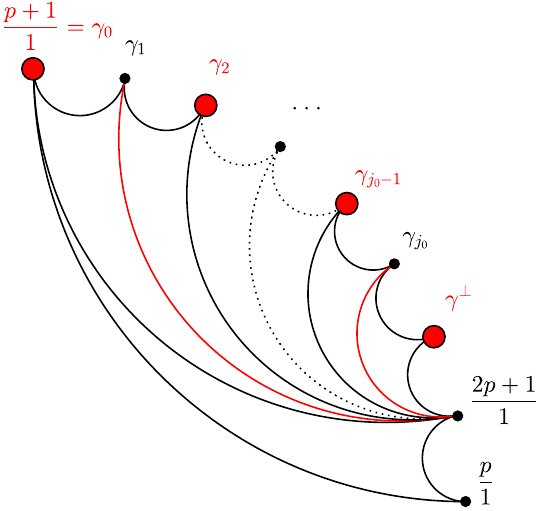}}
  \caption{Relevant part of the Farey tesselation in the first case (not drawn to scale). Even slopes and arcs adding saddles are marked in red. Left: $p$ is even, right: $p$ is odd.
    \label{fig:case1farey}}
\end{figure}

If $p$ is even, we have $d(\gamma^\perp, \mathcal{F}_e) = \frac{j_0}{2}+1= \frac{p-2m}{2}.$ Hence 
\[
 \frac{p-2m}{2} = d(\gamma^\perp, \mathcal{F}_e)  \ge \mu(\gamma) = p-2m
 \]
This is impossible since $p-2m>1.$

If $p$ is odd, we have $d(\gamma^\perp, \mathcal{F}_e) = \frac{j_0-1}{2}+1 = \frac{p-2m-1}{2}.$ Hence
\[
 \frac{p-2m-1}{2} = d(\gamma^\perp, \mathcal{F}_e)  \ge \mu(\gamma) = p-2m
 \]
This is also impossible since $p-2m>1.$

\textbf{Case 2:} Hence suppose that $p-2m < 0;$ equivalently $2m-p \ge 1.$ Then the normal arc coordinate of $\gamma$ is $(2m-p-1, 0, 1),$ and so $\mu(\gamma) = \max(1, 2m-p-1).$

\begin{figure}
  \centerline{\includegraphics[width=\textwidth]{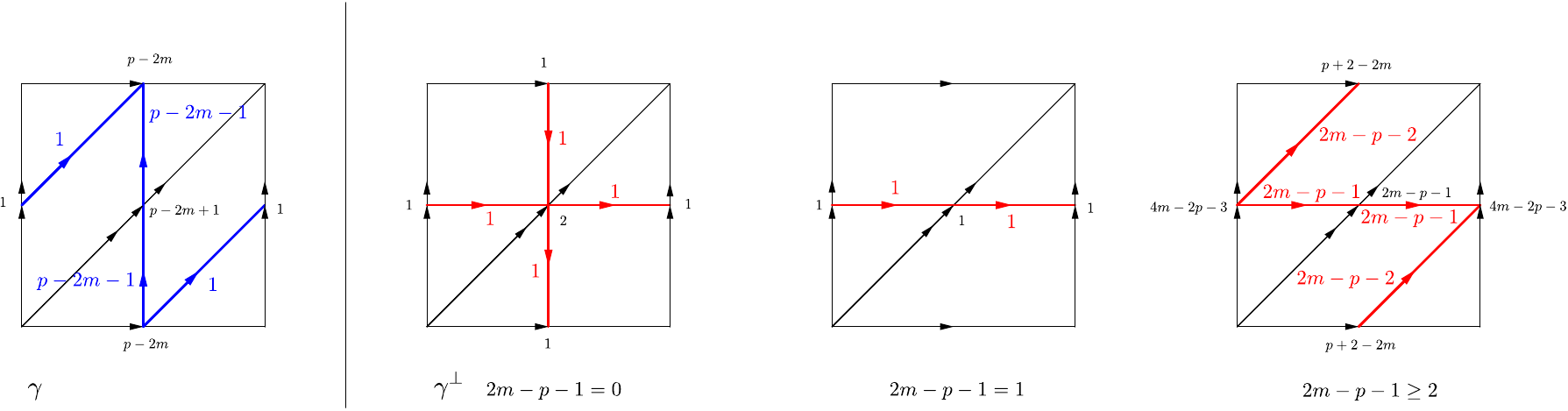}}
  \caption{Normal coordinate of $\gamma$ and its complementary curve in the second case
    \label{fig:case2complementary}}
\end{figure}

If $2m-p =1,$ then $\gamma^\perp = \meridian^{p+1}_g \longitude_g$ is a spanning slope. This is a contradiction.

If $2m-p =2,$ then $\gamma^\perp = \meridian^{p}_g \longitude_g$ is a spanning slope. This is a contradiction.

Hence $2m-p \ge 3$ and so $\mu(\gamma) = 2m-p-1.$ Then the complementary slope has normal arc coordinate $(0, 2m-p-1, 2m-p-2),$ and hence we have
\[
\langle \meridian_g, \gamma^{\perp} \rangle = 4m-2p-3, \quad 
\langle \meridian_g^{p}\longitude_g, \gamma^{\perp} \rangle = p+2-2m, \quad
\langle \meridian_g^{p+1}\longitude_g, \gamma^{\perp}\rangle = 2m-p-1
\]
This gives (again shown in additive notation):
\[
\gamma^\perp = (2m(1+2p)-2(p+1)^2) \; \meridian_g \; + \;  (4m-2p-3) \; \longitude_g
\]
Now
\[
\gamma^\perp  = p+1+\cfrac{1}{-2+\cfrac{1}{2m-p-1}} = p+1 + [\; -2,\; 2m-p-1\;]
\]
Let 
\[
\gamma_j  = \frac{p}{1} \oplus j \; \frac{2p+1}{2} 
\]
for $j \in \{0, \ldots, j_0\},$ where $j_0 = 2m-p-3.$
In particular, 
\[
\gamma_{j_0}  = \frac{4mp+2m-6p-2p^2-3}{4m-2p-5} 
\]
and we have  
\[
\frac{p+1}{1} > \frac{2p+1}{2} > \gamma^\perp > \gamma_{j_0} > \ldots > \gamma_2 > \gamma_1 >  \gamma_0 = \frac{p}{1}
\]
We observe that $\gamma^\perp = \gamma_{j_0} \oplus  \frac{2p+1}{2},$ and that 
$\det \begin{pmatrix} 4mp+2m-6p-2p^2-3 & 2p+1 \\ 4m-2p-5 & 2 \end{pmatrix} = -1,$ and so there is a triangle in the Farey tesselation with vertices $\gamma^\perp, \gamma_{j_0}$ and $\frac{2p+1}{2}.$ As above, this allows us to compute the slope distance of $\gamma^\perp$ to the nearest spanning slope from the triangles pivoting about $\frac{2p+1}{2}.$

\begin{figure}
  \centerline{\includegraphics[width=0.45\textwidth]{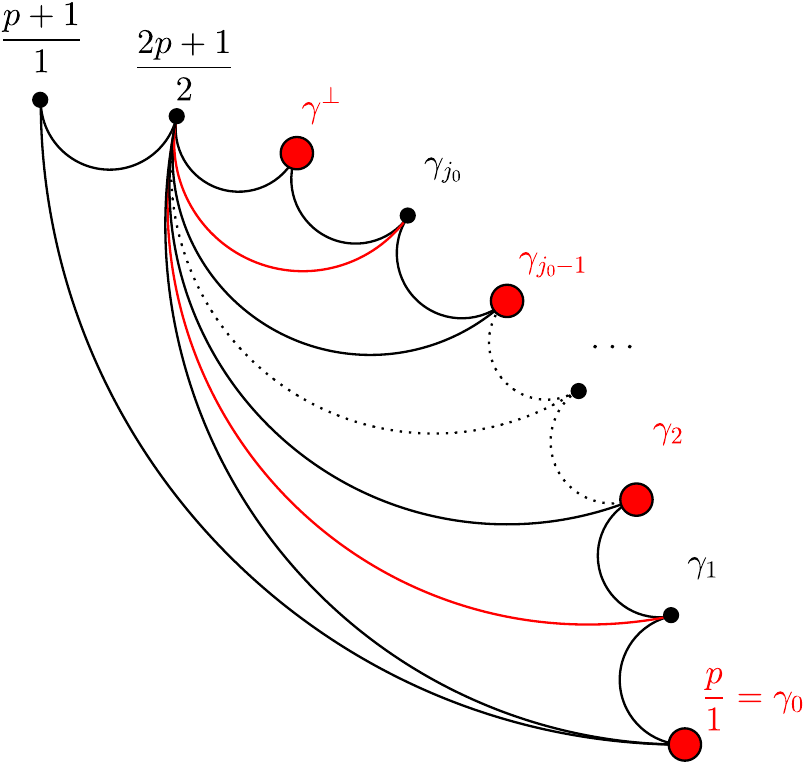} \,\,\, \includegraphics[width=0.45\textwidth]{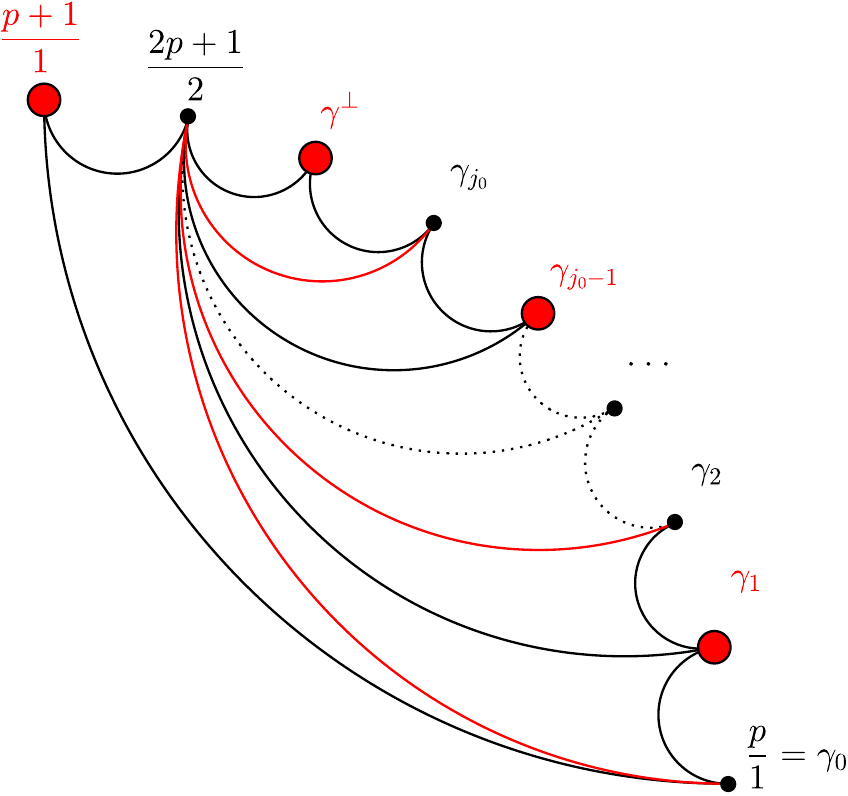}}
  \caption{Relevant part of the Farey tesselation in the second case (not drawn to scale). Even slopes and arcs adding saddles are marked in red. Left: $p$ is even, right: $p$ is odd.
    \label{fig:case2farey}}
\end{figure}

If $p$ is even, we have $d(\gamma^\perp, \mathcal{F}_e) = \frac{j_0-1}{2} + 1 = \frac{2m-p-2}{2}.$ Hence 
\[
 \frac{2m-p-2}{2} = d(\gamma^\perp, \mathcal{F}_e)  \ge \mu(\gamma) = 2m-p-1
 \]
This is impossible since $2m-p \ge 3.$

If $p$ is odd, we have $d(\gamma^\perp, \mathcal{F}_e) = \frac{j_0}{2} + 1 = \frac{2m-p-1}{2}.$ Hence
\[
 \frac{2m-p-1}{2} = d(\gamma^\perp, \mathcal{F}_e)  \ge \mu(\gamma) = 2m-p-1
 \]
This is also impossible since $2m-p \ge 3.$
 
Since in each case, the existence of $X'$ with complementary slope to $\partial F$ gives a contradition to the maximality of $\chi(F)$ amongst all spanning surfaces, this completes the proof.
\end{proof}


\section{Implementation and computational results}
\label{sec:computation}

According to the \texttt{KnotInfo} database \cite{knotinfo}, crosscap numbers are known for all knots with fewer than $10$ crossings. But there are five knots with ten, $96$ knots with $11$, and $668$ knots with $12$ crossings for which only bounds have been known for the crosscap number.

Here we present $196$ crosscap numbers of knots, for which crosscap numbers were previously unknown. This includes all five such $10$-crossing knots, $45$ of the missing $96$ crosscap numbers for $11$-crossing knots, and $146$ of the missing $668$ crosscap numbers for $12$-crossing knots. As a result, crosscap numbers for all knots up to ten crossings are now known.

\subsection{Implementation}
\label{sec:Implementation}

Our implementation uses out-of-the-box Regina functions. It is based on \Cref{cor:Q-procedure}, rather than the stronger \Cref{thm:Q-algo}, because we were also interested in the Euler characteristic of non-spanning surfaces with connected boundary.

\begin{algorithm}
\label{algo:crosscapno-Q}
Compute crosscap numbers using $Q$--coordinates and \Cref{cor:Q-procedure}.
\begin{algorithmic}
  \State {\sc Input:}
  \State 1. $\tri$ $0$--efficient suitable triangulation of $M$ with boundary $\tri_\partial$
  \State 2. $\meridian$ meridian of $K$ represented by an edge of $\tri_\partial$
  \State
  \State Compute $S_0$, the set of $Q$--fundamental surfaces of $\tri$
  \State Compute $S_1 = \{S\in S_0 \mid \partial S \neq \emptyset, \textrm{ connected, non-trivial} \}$, $Q$--fundamental surfaces with single essential boundary component
  \State Compute $B' = \min \{ 2- \chi(S) \mid S \in S_1 \textrm{ orientable}, |\partial S \cap \meridian | = 1 \}$
  \State Compute $A' = \min \{ 1- \chi(S) \mid S \in S_1 \textrm{ non-orientable}, |\partial S \cap \meridian | = 1 \}$
  \State Compute $N' = \min \{ 1- \chi(S) \mid S \in S_1, |\partial S \cap \meridian | > 1 \}$
  \If{$N' < \min ( A', B')$}
    \State \Return cannot determine crosscap number
  \Else
    \State \Return $\min (A', B')$
  \EndIf 
\end{algorithmic}
\end{algorithm}

The main computational effort is in Regina's enumeration algorithm for $Q$--fundamental surfaces \cite{regina}, which in turn runs a Hilbert basis enumeration on a high-dimensional polytope. The verification of the correctness of the input also takes up 
significant -- but smaller amounts of -- computational resources. The first verification is the test for $0$-efficiency of the triangulation. The second verifies the meridian edge. Here, we perform a Dehn surgery along this edge, and then use Regina's 3--sphere recognition routine to check that the resulting 3--manifold is indeed the 3--sphere.

\subsection{Computational results}

In supporting material for \cite{burton12-crosscap}, Burton and Ozlen compiled a list of triangulations of all knot complements up to $12$ crossings that are $0$-efficient, with real boundary, and one of the boundary edges running parallel to the meridian $\meridian$ (i.e., $0$-efficient, suitable triangulations). This list is available from the webpage of the first author. Using this list of triangulations, we applied the implementation outlined in \Cref{sec:Implementation}.

The results are summarised in \Cref{tab:newcc1011,tab:newcc12}. 
Here, ``$nOr$" is the maximal Euler characteristic of a non-orientable $Q$--fundamental spanning surface, ``$or$" that of an orientable $Q$--fundamental spanning surface, and ``$nSp$" that of a $Q$--fundamental non-spanning surface with single boundary component.
As an additional check, we also ran our algorithm for a larger collection of knots for which computations are feasible.

We ran our computations on a machine with $2 \times 24$ Intel Xeon Gold6240R processors and 192GB of memory. Computations were feasible for triangulations of up to $30$ tetrahedra (on a standard laptop, triangulations up to $27$ tetrahedra can still be handled). We used roughly six months of CPU time to obtain the data in \Cref{tab:newcc1011,tab:newcc12}. We only tried Regina's default choice of Hilbert basis algorithm.

\begin{table}[h]
\caption{New $10$- and $11$-crossing crosscap numbers.\label{tab:newcc1011}} 

\begin{center}
\begin{tabular}{l@{\,\,\,\,}c@{\,\,\,\,}lll|l@{\,\,\,\,}c@{\,\,\,\,}lll|l@{\,\,\,\,}c@{\,\,\,\,}lll}

DT  & $\cross(K)$ & nOr & or & nSp & DT  & $\cross(K)$ & nOr & or & nSp & DT  & $\cross(K)$ & nOr & or & nSp \\ \hline 

$10_{157}$&$4$&$-3$&$-5$&$-7$&
$10_{159}$&$4$&$-3$&$-5$&$-4$&
$10_{164}$&$4$&$-3$&$-3$&$-6$\\
$10_{158}$&$4$&$-3$&$-5$&$-6$&
$10_{163}$&$4$&$-3$&$-5$&$-6$&&&&&\\

\multicolumn{15}{r}{} \\ 

$11n_{2}$&$4$&$-3$&$-5$&$-6$&
$11n_{59}$&$4$&$-3$&$-5$&$-5$&
$11n_{120}$&$4$&$-3$&$-7$&$-6$\\
$11n_{3}$&$4$&$-3$&$-3$&$-5$&
$11n_{75}$&$4$&$-3$&$-5$&$-6$&
$11n_{121}$&$4$&$-3$&$-5$&$-6$\\
$11n_{4}$&$4$&$-3$&$-5$&$-6$&
$11n_{76}$&$3$&$-2$&$-7$&$-6$&
$11n_{123}$&$4$&$-3$&$-3$&$-7$\\
$11n_{7}$&$4$&$-3$&$-5$&$-6$&
$11n_{77}$&$4$&$-3$&$-7$&$-5$&
$11n_{124}$&$4$&$-3$&$-5$&$-6$\\
$11n_{11}$&$4$&$-3$&$-5$&$-7$&
$11n_{78}$&$3$&$-2$&$-7$&$-5$&
$11n_{130}$&$4$&$-3$&$-5$&$-8$\\
$11n_{22}$&$4$&$-3$&$-5$&$-6$&
$11n_{83}$&$4$&$-3$&$-3$&$-6$&
$11n_{134}$&$4$&$-3$&$-3$&$-4$\\
$11n_{25}$&$4$&$-3$&$-5$&$-6$&
$11n_{86}$&$4$&$-3$&$-5$&$-6$&
$11n_{137}$&$4$&$-3$&$-5$&$-5$\\
$11n_{29}$&$4$&$-3$&$-3$&$-5$&
$11n_{87}$&$4$&$-3$&$-5$&$-7$&
$11n_{158}$&$4$&$-3$&$-7$&$-5$\\
$11n_{33}$&$4$&$-3$&$-5$&$-6$&
$11n_{89}$&$4$&$-3$&$-5$&$-5$&
$11n_{162}$&$4$&$-3$&$-3$&$-6$\\
$11n_{39}$&$4$&$-3$&$-3$&$-5$&
$11n_{93}$&$4$&$-3$&$-5$&$-5$&
$11n_{164}$&$4$&$-3$&$-5$&$-6$\\
$11n_{45}$&$4$&$-3$&$-5$&$-6$&
$11n_{100}$&$4$&$-3$&$-3$&$-5$&
$11n_{170}$&$4$&$-3$&$-3$&$-6$\\
$11n_{47}$&$4$&$-3$&$-7$&$-6$&
$11n_{109}$&$4$&$-3$&$-5$&$-6$&
$11n_{172}$&$4$&$-3$&$-5$&$-5$\\
$11n_{52}$&$4$&$-3$&$-5$&$-7$&
$11n_{112}$&$4$&$-3$&$-5$&$-6$&
$11n_{173}$&$4$&$-3$&$-7$&$-6$\\
$11n_{54}$&$4$&$-3$&$-5$&$-5$&
$11n_{114}$&$4$&$-3$&$-3$&$-7$&
$11n_{175}$&$4$&$-3$&$-5$&$-5$\\
$11n_{55}$&$4$&$-3$&$-5$&$-6$&
$11n_{117}$&$3$&$-2$&$-3$&$-4$&
$11n_{180}$&$4$&$-3$&$-5$&$-6$\\

\end{tabular}
\end{center}
\end{table}

\begin{table}
\caption{New $12$-crossing crosscap numbers. \label{tab:newcc12}}

\begin{center}
\begin{tabular}{l@{\,\,\,\,}c@{\,\,\,\,}lll|l@{\,\,\,\,}c@{\,\,\,\,}lll|l@{\,\,\,\,}c@{\,\,\,\,}lll}

DT  & $\cross(K)$ & nOr & or & nSp & DT  & $\cross(K)$ & nOr & or & nSp & DT  & $\cross(K)$ & nOr & or & nSp \\ \hline 
$12n_{1}$&$5$&$-4$&$-5$&$-6$&
$12n_{195}$&$3$&$-2$&$-5$&$-6$&
$12n_{423}$&$4$&$-3$&$-5$&$-7$\\
$12n_{7}$&$4$&$-3$&$-5$&$-5$&
$12n_{203}$&$4$&$-3$&$-5$&$-5$&
$12n_{425}$&$5$&$-4$&$-7$&$-6$\\
$12n_{8}$&$4$&$-3$&$-7$&$-6$&
$12n_{210}$&$4$&$-3$&$-5$&$-6$&
$12n_{430}$&$4$&$-3$&$-3$&$-4$\\
$12n_{10}$&$4$&$-3$&$-5$&$-6$&
$12n_{211}$&$4$&$-3$&$-5$&$-7$&
$12n_{437}$&$4$&$-3$&$-5$&$-6$\\
$12n_{11}$&$4$&$-3$&$-3$&$-6$&
$12n_{215}$&$3$&$-2$&$-5$&$-6$&
$12n_{442}$&$4$&$-3$&$-3$&$-5$\\
$12n_{16}$&$3$&$-2$&$-7$&$-5$&
$12n_{220}$&$4$&$-3$&$-7$&$-6$&
$12n_{452}$&$4$&$-3$&$-3$&$-6$\\
$12n_{19}$&$3$&$-2$&$-5$&$-6$&
$12n_{225}$&$4$&$-3$&$-3$&$-8$&
$12n_{469}$&$3$&$-2$&$-5$&$-6$\\
$12n_{20}$&$5$&$-4$&$-5$&$-6$&
$12n_{229}$&$4$&$-3$&$-7$&$-5$&
$12n_{476}$&$4$&$-3$&$-5$&$-5$\\
$12n_{24}$&$4$&$-3$&$-5$&$-6$&
$12n_{230}$&$3$&$-2$&$-5$&$-5$&
$12n_{479}$&$4$&$-3$&$-3$&$-7$\\
$12n_{38}$&$5$&$-4$&$-5$&$-7$&
$12n_{237}$&$4$&$-3$&$-5$&$-7$&
$12n_{484}$&$4$&$-3$&$-5$&$-6$\\
$12n_{40}$&$5$&$-4$&$-7$&$-7$&
$12n_{241}$&$4$&$-3$&$-5$&$-5$&
$12n_{494}$&$4$&$-3$&$-7$&$-5$\\
$12n_{42}$&$4$&$-3$&$-5$&$-5$&
$12n_{247}$&$3$&$-2$&$-3$&$-5$&
$12n_{495}$&$4$&$-3$&$-5$&$-7$\\
$12n_{43}$&$4$&$-3$&$-7$&$-7$&
$12n_{257}$&$3$&$-2$&$-5$&$-5$&
$12n_{509}$&$4$&$-3$&$-7$&$-5$\\
$12n_{45}$&$5$&$-4$&$-5$&$-6$&
$12n_{261}$&$4$&$-3$&$-7$&$-6$&
$12n_{526}$&$4$&$-3$&$-7$&$-7$\\
$12n_{51}$&$4$&$-3$&$-3$&$-6$&
$12n_{271}$&$4$&$-3$&$-5$&$-6$&
$12n_{535}$&$4$&$-3$&$-3$&$-6$\\
$12n_{53}$&$4$&$-3$&$-3$&$-7$&
$12n_{274}$&$4$&$-3$&$-3$&$-6$&
$12n_{547}$&$4$&$-3$&$-3$&$-7$\\
$12n_{56}$&$4$&$-3$&$-5$&$-6$&
$12n_{276}$&$4$&$-3$&$-5$&$-6$&
$12n_{552}$&$3$&$-2$&$-5$&$-6$\\
$12n_{63}$&$4$&$-3$&$-5$&$-6$&
$12n_{278}$&$4$&$-3$&$-3$&$-7$&
$12n_{554}$&$4$&$-3$&$-3$&$-7$\\
$12n_{64}$&$4$&$-3$&$-7$&$-5$&
$12n_{279}$&$4$&$-3$&$-3$&$-7$&
$12n_{566}$&$4$&$-3$&$-3$&$-5$\\
$12n_{67}$&$4$&$-3$&$-7$&$-6$&
$12n_{280}$&$4$&$-3$&$-5$&$-6$&
$12n_{572}$&$4$&$-3$&$-5$&$-5$\\
$12n_{68}$&$4$&$-3$&$-7$&$-6$&
$12n_{285}$&$4$&$-3$&$-5$&$-6$&
$12n_{573}$&$4$&$-3$&$-5$&$-5$\\
$12n_{71}$&$4$&$-3$&$-7$&$-4$&
$12n_{290}$&$4$&$-3$&$-5$&$-7$&
$12n_{580}$&$4$&$-3$&$-3$&$-6$\\
$12n_{73}$&$4$&$-3$&$-5$&$-6$&
$12n_{304}$&$4$&$-3$&$-5$&$-6$&
$12n_{585}$&$4$&$-3$&$-5$&$-6$\\
$12n_{74}$&$4$&$-3$&$-7$&$-6$&
$12n_{308}$&$4$&$-3$&$-5$&$-5$&
$12n_{601}$&$4$&$-3$&$-5$&$-6$\\
$12n_{78}$&$4$&$-3$&$-3$&$-6$&
$12n_{311}$&$4$&$-3$&$-3$&$-6$&
$12n_{605}$&$4$&$-3$&$-7$&$-6$\\
$12n_{82}$&$4$&$-3$&$-5$&$-7$&
$12n_{312}$&$4$&$-3$&$-5$&$-7$&
$12n_{607}$&$3$&$-2$&$-5$&$-6$\\
$12n_{84}$&$4$&$-3$&$-5$&$-7$&
$12n_{324}$&$4$&$-3$&$-3$&$-7$&
$12n_{610}$&$4$&$-3$&$-7$&$-7$\\
$12n_{89}$&$4$&$-3$&$-7$&$-7$&
$12n_{327}$&$4$&$-3$&$-7$&$-6$&
$12n_{623}$&$4$&$-3$&$-7$&$-6$\\
$12n_{93}$&$4$&$-3$&$-7$&$-6$&
$12n_{331}$&$3$&$-2$&$-5$&$-4$&
$12n_{630}$&$4$&$-3$&$-5$&$-5$\\
$12n_{97}$&$4$&$-3$&$-5$&$-6$&
$12n_{334}$&$4$&$-3$&$-3$&$-5$&
$12n_{641}$&$4$&$-3$&$-7$&$-6$\\
$12n_{104}$&$4$&$-3$&$-7$&$-6$&
$12n_{341}$&$4$&$-3$&$-5$&$-6$&
$12n_{642}$&$4$&$-3$&$-3$&$-4$\\
$12n_{106}$&$4$&$-3$&$-7$&$-5$&
$12n_{342}$&$4$&$-3$&$-3$&$-5$&
$12n_{643}$&$3$&$-2$&$-5$&$-7$\\
$12n_{116}$&$4$&$-3$&$-5$&$-6$&
$12n_{343}$&$4$&$-3$&$-3$&$-5$&
$12n_{650}$&$4$&$-3$&$-3$&$-5$\\
$12n_{124}$&$4$&$-3$&$-3$&$-6$&
$12n_{345}$&$4$&$-3$&$-7$&$-6$&
$12n_{674}$&$4$&$-3$&$-7$&$-6$\\
$12n_{129}$&$4$&$-3$&$-5$&$-6$&
$12n_{354}$&$4$&$-3$&$-5$&$-5$&
$12n_{688}$&$4$&$-3$&$-9$&$-6$\\
$12n_{134}$&$4$&$-3$&$-7$&$-5$&
$12n_{355}$&$3$&$-2$&$-5$&$-5$&
$12n_{699}$&$4$&$-3$&$-3$&$-6$\\
$12n_{146}$&$4$&$-3$&$-3$&$-6$&
$12n_{359}$&$4$&$-3$&$-3$&$-7$&
$12n_{709}$&$4$&$-3$&$-7$&$-5$\\
$12n_{150}$&$4$&$-3$&$-7$&$-6$&
$12n_{360}$&$4$&$-3$&$-3$&$-6$&
$12n_{718}$&$4$&$-3$&$-5$&$-6$\\
$12n_{152}$&$4$&$-3$&$-5$&$-6$&
$12n_{362}$&$4$&$-3$&$-5$&$-6$&
$12n_{719}$&$4$&$-3$&$-5$&$-6$\\
$12n_{154}$&$4$&$-3$&$-5$&$-5$&
$12n_{366}$&$4$&$-3$&$-5$&$-6$&
$12n_{726}$&$4$&$-3$&$-3$&$-5$\\
$12n_{160}$&$4$&$-3$&$-5$&$-6$&
$12n_{377}$&$4$&$-3$&$-5$&$-5$&
$12n_{730}$&$4$&$-3$&$-5$&$-6$\\
$12n_{162}$&$4$&$-3$&$-5$&$-5$&
$12n_{379}$&$4$&$-3$&$-5$&$-6$&
$12n_{739}$&$3$&$-2$&$-7$&$-5$\\
$12n_{170}$&$4$&$-3$&$-3$&$-6$&
$12n_{381}$&$4$&$-3$&$-3$&$-6$&
$12n_{764}$&$4$&$-3$&$-5$&$-5$\\
$12n_{179}$&$4$&$-3$&$-5$&$-7$&
$12n_{383}$&$4$&$-3$&$-3$&$-5$&
$12n_{797}$&$4$&$-3$&$-3$&$-6$\\
$12n_{185}$&$4$&$-3$&$-7$&$-5$&
$12n_{388}$&$4$&$-3$&$-5$&$-5$&
$12n_{808}$&$4$&$-3$&$-5$&$-5$\\
$12n_{187}$&$4$&$-3$&$-7$&$-7$&
$12n_{390}$&$4$&$-3$&$-5$&$-6$&
$12n_{824}$&$4$&$-3$&$-5$&$-7$\\
$12n_{188}$&$4$&$-3$&$-7$&$-6$&
$12n_{397}$&$4$&$-3$&$-5$&$-6$&
$12n_{870}$&$3$&$-2$&$-5$&$-5$\\
$12n_{190}$&$4$&$-3$&$-7$&$-5$&
$12n_{407}$&$3$&$-2$&$-5$&$-5$&
$12n_{884}$&$4$&$-3$&$-3$&$-7$\\
$12n_{193}$&$3$&$-2$&$-5$&$-6$&
$12n_{418}$&$3$&$-2$&$-7$&$-5$&

\end{tabular}
\end{center}
\end{table}

\subsection{Surfaces realising crosscap number}

One interesting aspect of our algorithms is the relationship between our numbers $A$ and $B$ from  \Cref{thm:crosscap-suitable} and $A'$ and $B'$ from \Cref{thm:Q-algo}. We have $\min(A, B) =c(K) = \min(A', B')$.  Since $Q$--fundamental surfaces are fundamental surfaces, we also have $A\le A'$ and $B\le B'.$ 

We have the following observations from our computations:

\begin{enumerate}
  \item If $A' < B'$, then $A' = A \le B \le B'.$ 
  This is the most common case for small crossing knots.  Smallest examples are the trefoil and the figure-$8$ knot. Moreover, the gap between $A'$ and $B'$ can be arbitrarily large, as can be seen from the torus knots $T(2,2k+1)$, $k \geq 1$: the crosscap number of these knots is $1$, whereas the knot genus is $k$. In other words, we have $A' = A=1< 2k+1=B\le B'$.
  \item If $A'=B'$, then $A'=A = B =B'.$ Here, the crosscap number is realised by a $\partial$-compressible surface obtained from a minimum genus Seifert surface with a M\"obius band attached, but the existence of a $\partial$-incompressible non-orientable spanning surface realising it is not excluded. Smallest knots with this property are $7a_6$ ($7_4$), $8a_{18}$ ($8_3$). All such knots must necessarily have genus $k$ and crosscap number $2k+1$. It is worthwhile mentioning that $7a_6$ is known not to admit a $\partial$-incompressible non-orientable spanning surface realising the crosscap number, see \cite{ICHIHARA2002467} and the references therein.
  \item If $A'>B',$ then $A' \ge A \ge B = B'.$ Interesting examples from our calculations are:
\begin{enumerate}
\item  
  The case $A'>A=B'=B$ occurred for the eleven crossing knot $11a_{362}$ of genus one and crosscap number three, where $A'=4> 3 = A=B=B'$ for the suitable triangulation with Regina isomorphism signature
  
   \centerline{\texttt{uLLvMPvwMwAMQkcacfgihjmklnnrqstrqrtnkvjhavkbveekgjxfcvp}. }
   
   In standard coordinates, fundamental normal surfaces realising the spanning punctured torus $T$, and a spanning non-orientable surface $S$ of Euler characteristic $-2$ have complementary boundary slopes, and $\mu(\partial T)=1=\mu(\partial S)$. In quadrilateral coordinates, we have $S + D = A + T$, where $A$ is a fundamental boundary parallel annulus with boundary curves parallel to $\partial S$, and $D$ is the vertex linking disk. In particular, using the notation from \Cref{sec:quad_space}, we have $k_1=1$, $k_0=0$, ${F_i} = {A}$, ${G_j} = {S}$, and $k'_0 = k''_0 = 0$. 
  \item The case $A'=A > B = B'$ occurred for the eleven crossing knot $11n_{139}$. This knot has genus one, and crosscap number three, but $A'=A=4> 3 = B = B'$.   
  \end{enumerate}
  \item In all examples, where we computed all of $A, A', B, B'$, we observed $B = B'.$
\end{enumerate}


\bibliographystyle{plain}
\bibliography{references}

\begin{thebibliography}{10}

\bibitem{Adams13CrosscapNoForAlternatingKnots}
Colin Adams and Thomas Kindred.
\newblock A classification of spanning surfaces for alternating links.
\newblock {\em Algebr. Geom. Topol.}, 13(5):2967--3007, 2013.

\bibitem{bachman16complicatedHeegaardSplittings}
David Bachman, Ryan Derby-Talbot, and Eric Sedgwick.
\newblock {Heegaard structure respects complicated JSJ decompositions}.
\newblock {\em Mathematische Annalen}, 365:1137--1154, 2016.

\bibitem{regina}
Benjamin~A. Burton, Ryan Budney, William Pettersson, et~al.
\newblock Regina: Software for low-dimensional topology.
\newblock {\tt http://\allowbreak regina-normal.\allowbreak github.\allowbreak
  io/}, 1999--2021.

\bibitem{burton12-crosscap}
Benjamin~A. Burton and Melih Ozlen.
\newblock Computing the crosscap number of a knot using integer programming and
  normal surfaces.
\newblock {\em ACM Trans. Math. Software}, 39:4:1--4:18, 2012.

\bibitem{Burton-computing-2018}
Benjamin~A. Burton and Stephan Tillmann.
\newblock Computing closed essential surfaces in 3-manifolds.
\newblock https://arxiv.org/abs/1812.11686, 2018.

\bibitem{Clark78CrosscapNumber}
Bradd~E. Clark.
\newblock Crosscaps and knots.
\newblock {\em Internat. J. Math. Math. Sci.}, 1:113--123, 1978.

\bibitem{Hatcher-boundary-1982}
Allen~E. Hatcher.
\newblock On the boundary curves of incompressible surfaces.
\newblock {\em Pacific J. Math.}, 99(2):373--377, 1982.

\bibitem{HatcherNotes}
Allen~E. Hatcher.
\newblock {Notes on Basic 3-Manifold Topology}.
\newblock https://pi.math.cornell.edu/~hatcher/3M/3M.pdf, 2000.

\bibitem{Howie-geography}
Joshua Howie.
\newblock Geography of spanning surfaces.
\newblock in preparation, $\ge$2021.

\bibitem{ICHIHARA2002467}
Kazuhiro Ichihara, Masahiro Ohtouge, and Masakazu Teragaito.
\newblock Boundary slopes of non-orientable {S}eifert surfaces for knots.
\newblock {\em Topology and its Applications}, 122(3):467--478, 2002.

\bibitem{Ito18CCKnotProjections}
Noboru Ito and Yusuke Takimura.
\newblock Crosscap number and knot projections.
\newblock {\em Internat. J. Math.}, 29(12):1850084, 21, 2018.

\bibitem{Ito20CCAndVolumeBounds}
Noboru Ito and Yusuke Takimura.
\newblock Crosscap number of knots and volume bounds.
\newblock {\em Internat. J. Math.}, 31(13):2050111, 33, 2020.

\bibitem{Ito20CCBoundsAlternatingKnots}
Noboru Ito and Yusuke Takimura.
\newblock A lower bound of crosscap numbers of alternating knots.
\newblock {\em J. Knot Theory Ramifications}, 29(1):1950092, 15, 2020.

\bibitem{Jaco-algorithm-1984}
William Jaco and Ulrich Oertel.
\newblock An algorithm to decide if a {$3$}-manifold is a {H}aken manifold.
\newblock {\em Topology}, 23(2):195--209, 1984.

\bibitem{jaco03-0-efficiency}
William Jaco and J.~Hyam Rubinstein.
\newblock 0-efficient triangulations of 3-manifolds.
\newblock {\em J. Differential Geom.}, 65(1):61--168, 2003.

\bibitem{Jaco-norm-2020}
William Jaco, J.~Hyam Rubinstein, Jonathan Spreer, and Stephan Tillmann.
\newblock {$\Bbb{Z}_2$}-{T}hurston norm and complexity of 3-manifolds, {II}.
\newblock {\em Algebr. Geom. Topol.}, 20(1):503--529, 2020.

\bibitem{Jaco-ideal-2020}
William Jaco, J.~Hyam Rubinstein, Jonathan Spreer, and Stephan Tillmann.
\newblock On minimal ideal triangulations of cusped hyperbolic 3-manifolds.
\newblock {\em J. Topol.}, 13(1):308--342, 2020.

\bibitem{Jaco-minimal-2009}
William Jaco, J.~Hyam Rubinstein, and Stephan Tillmann.
\newblock Minimal triangulations for an infinite family of lens spaces.
\newblock {\em J. Topol.}, 2(1):157--180, 2009.

\bibitem{Jaco-decision-2003}
William Jaco and Eric Sedgwick.
\newblock Decision problems in the space of {D}ehn fillings.
\newblock {\em Topology}, 42(4):845--906, 2003.

\bibitem{Kalfagianni16BoundsForCrosscapNumber}
Efstratia Kalfagianni and Christine Ruey~Shan Lee.
\newblock Crosscap numbers and the {J}ones polynomial.
\newblock {\em Adv. Math.}, 286:308--337, 2016.

\bibitem{knotinfo}
Charles Livingston and Allison~H. Moore.
\newblock Knotinfo: Table of knot invariants.
\newblock URL: \url{knotinfo.math.indiana.edu}, August 2021.
\newblock Page on crosscap numbers:
  \url{https://knotinfo.math.indiana.edu/descriptions/crosscap_number.html}.

\bibitem{Matveev-algorithmic-2007}
Sergei Matveev.
\newblock {\em Algorithmic topology and classification of 3-manifolds},
  volume~9 of {\em Algorithms and Computation in Mathematics}.
\newblock Springer, Berlin, second edition, 2007.

\bibitem{Schubert1961-bestimmung}
Horst Schubert.
\newblock Bestimmung der {P}rimfaktorzerlegung von {V}erkettungen.
\newblock {\em Math. Z.}, 76:116--148, 1961.

\bibitem{Stallings-fibering-1962}
John Stallings.
\newblock On fibering certain {$3$}-manifolds.
\newblock In {\em Topology of 3-manifolds and related topics ({P}roc. {T}he
  {U}niv. of {G}eorgia {I}nstitute, 1961)}, pages 95--100. Prentice-Hall,
  Englewood Cliffs, N.J., 1962.

\bibitem{Teragaito04CCTorusKnots}
Masakazu Teragaito.
\newblock Crosscap numbers of torus knots.
\newblock {\em Topology Appl.}, 138(1-3):219--238, 2004.

\bibitem{Tillus-normal-2008}
Stephan Tillmann.
\newblock Normal surfaces in topologically finite 3-manifolds.
\newblock {\em Enseign. Math. (2)}, 54(3-4):329--380, 2008.

\bibitem{tollefson98-quadspace}
Jeffrey~L. Tollefson.
\newblock Normal surface {$Q$}-theory.
\newblock {\em Pacific J. Math.}, 183(2):359--374, 1998.

\end{thebibliography}

\address{William Jaco\\Department of Mathematics, Oklahoma State University, Stillwater, OK 74078-1058, USA\\{jaco@math.okstate.edu}\\-----}

\address{J. Hyam Rubinstein\\School of Mathematics and Statistics, The University of Melbourne, VIC 3010, Australia\\
{joachim@unimelb.edu.au}\\----- }

\address{Jonathan Spreer\\School of Mathematics and Statistics F07, The University of Sydney, NSW 2006 Australia\\{jonathan.spreer@sydney.edu.au\\-----}}

\address{Stephan Tillmann\\School of Mathematics and Statistics F07, The University of Sydney, NSW 2006 Australia\\{stephan.tillmann@sydney.edu.au}}

\Addresses
                                                      
\end{document}